\documentclass[a4paper,fleqn]{cas-sc}

\usepackage[authoryear,longnamesfirst]{natbib}
\usepackage{graphicx}%
\usepackage{multirow}%
\usepackage{amsmath,amssymb,amsfonts}%
\usepackage{amsthm}%
\usepackage{mathrsfs}%
\usepackage[title]{appendix}%
\usepackage{xcolor}%
\usepackage{textcomp}%
\usepackage{manyfoot}%
\usepackage{booktabs}%
\usepackage{algorithm}%
\usepackage{algorithmicx}%
\usepackage{algpseudocode}%
\usepackage{listings}%
\usepackage{ulem}
\usepackage{subfig}
\usepackage{colortbl}
\usepackage{multirow}
\usepackage{ulem}
\def\tsc#1{\csdef{#1}{\textsc{\lowercase{#1}}\xspace}}
\tsc{WGM}
\tsc{QE}
\tsc{EP}
\tsc{PMS}
\tsc{BEC}
\tsc{DE}

\newtheorem{lemma}{Lemma}
\newtheorem{theorem}{Theorem}
\newtheorem{remark}{Remark}%

\begin{document}
\let\WriteBookmarks\relax
\def\floatpagepagefraction{1}
\def\textpagefraction{.001}
\shorttitle{Convergence of Substructuring Waveform Relaxation Algorithms for Hyperbolic PDEs with Time Delay}
\shortauthors{Bankim C. Mandal and Deeksha Tomer}

\title [mode = title]{Convergence of Substructuring Waveform Relaxation Algorithms for Hyperbolic PDEs with Time Delay}                      
\author[inst1]{Bankim C. Mandal}[orcid=0009-0009-0134-0422] 
\ead{bmandal@iitbbs.ac.in}

\author[inst1]{Deeksha Tomer\cormark[1]}[orcid=0009-0002-3269-883X]
\ead{a21ma09002@iitbbs.ac.in}

\cortext[cor1]{Corresponding author}

\affiliation[inst1]{organization={School of Basic Sciences},
            addressline={Indian Institute of Technology Bhubaneswar}, 
            city={Odisha},
            postcode={752050}, 
            country={India}}

\begin{abstract}
This article investigates the application and analysis of two substructuring waveform relaxation algorithms namely Dirichlet-Neumann Waveform Relaxation (DNWR) and Neumann-Neumann Waveform Relaxation (NNWR) for solving hyperbolic partial differential equations (PDEs) with time delay. These equations are relevant in numerous physical and engineering contexts, such as wave propagation, biological processes, and control systems, where the system's dynamics are influenced by past states. The study emphasizes the stability, convergence, and computational efficiency of these non-overlapping domain decomposition methods when applied to such problems. Specifically, the DNWR and NNWR algorithms are analyzed using both Fourier and Laplace transforms in asymmetric domain decomposition to assess their capability to manage delayed terms in hyperbolic systems. Using Fourier analysis, we establish linear convergence estimate for the numerical errors. Laplace transform analysis enables a more in-depth study for characterizing finite-step convergence. Additionally, we derive the optimal parameters required to achieve finite step convergence in presence of heterogeneous spatial domain. Theoretical findings are complemented by numerical experiments, showcasing the methods' effectiveness in maintaining accuracy while reducing computational complexity. Additionally, the study explores potential extensions to more complex problems and diverse applications.
\end{abstract}




\maketitle

\section{Introduction}
Hyperbolic PDEs with time delay are fundamental mathematical models used to describe dynamic systems in which the system’s evolution is influenced by both its current state as well as its previous states. Such equations arise across a broad range of scientific and engineering problems, including wave propagation in viscoelastic materials \citep{hale1977}, control systems with delayed feedback \citep{gu2003}, and biological processes such as population dynamics and neural networks \citep{cooke1963differential}. The inclusion of time delays in these systems introduces additional complexity, making their analysis and numerical solution both challenging and essential for understanding real-world phenomena.

Domain decomposition methods (DDMs) are powerful parallel computational tools for solving large-scale PDEs, achieved by partitioning the computational domain into a collection of smaller, more manageable subdomains \citep{quarteroni1999}. DDMs are generally classified into two categories: overlapping and non-overlapping. Non-overlapping DDMs are often referred to as substructuring methods. When applied to time-dependent problems, these are known as substructuring waveform relaxation methods. Among them, Dirichlet-Neumann Waveform Relaxation (DNWR) and Neumann-Neumann Waveform Relaxation (NNWR) have gained significant attention due to their efficiency in handling both linear and non-linear time-dependent problems \citep{gobinda,sana2023dirichlet}. These methods iteratively solve subproblems on non-overlapping subdomains by imposing appropriate boundary conditions, such as Dirichlet or Neumann conditions, at the interfaces. While DNWR and NNWR have been extensively studied for parabolic and hyperbolic PDEs \citep{dd22,gander2016parabolic} their application to PDEs with time delay remains relatively unexplored, despite the growing interest in delay-dependent systems \citep{michiels2007,zhong2006}.

Another popular version of domain decomposition methods, Schwarz Waveform Relaxation (SWR) method \citep{gander1998space, gander2002overlapping} and its optimized counterpart \citep{gander1999optimal, halpern2012optimized} gain substantial momentum over the years for its ability to exploit parallelism. In problems involving time delay, where the computational cost can be high due to the coupling between past and future states, the SWR method allows for the distribution of subproblems across multiple processors. This significantly improves computational efficiency, particularly for large-scale systems with significant delay terms. Kwok, Ong and Mandal \citep{kwok2019schwarz} have provided insights into the parallel nature of the DNWR, NNWR and SWR method and its scalability for large systems. Furthermore, the articles \citet{shulin1, shulin2} address the application of SWR to various versions of delay differential equations and provide a convergence analysis.\\
A comprehensive analysis of NNWR and DNWR in the context of hyperbolic delay problems is currently lacking. To bridge this gap, this paper makes the following primary contributions:
\begin{itemize}
    \item This study investigates the behavior of DNWR method applied to hyperbolic PDEs with time delays under asymmetric domain decomposition. We establish linear convergence bound utilizing Fourier analysis; Laplace transform analysis demonstrates that the method achieves finite-step convergence.
    \item We derive the optimal parameter for achieving finite step convergence of the DNWR algorithm in presence of heterogeneous spatial domain.
    \item Finite-step convergence of NNWR algorithm is mathematically established for multi-subdomain configurations in 1D and 2D spatial settings.
\end{itemize}
Additionally, this framework can be extended to more intricate settings, including systems with multiple or distributed delays, nonlinear dynamics, and coupled partial differential equations, which constitute promising directions for future research \citep{app1,app2}. These methods have been studied and implemented for heterogeneous structures in recent times for various problems, see for examples \cite{Monge2018AMN,MongBir,sana2023dirichlet,gobinda2}. 
To the best of our knowledge, this is the first work addressing the rigorous analysis of DNWR and NNWR methods applied to delay problems.

The main body of the paper is arranged as follows: Sections 2 and 3 present the mathematical formulation of DNWR algorithm for hyperbolic PDEs with time delay including the heterogeneous case. In Section 4 we present the convergence analysis of NNWR method for multiple subdomain setting. Section 5 explores the possible extension of NNWR method in 2D. The numerical experiments and other computational outcomes are enlisted in Section 6. In Section 6, we also compare the performance of DNWR and NNWR with the classical SWR method, demonstrating the superior iteration efficiency of the proposed methods. In particular, DNWR and NNWR do not require overlap for convergence, which makes them well-suited for problems with heterogeneous coefficients, whereas the classical Schwarz method typically requires an overlap to ensure convergence.  
Through this study, we aim to advance the understanding and application of domain decomposition methods for hyperbolic PDEs with time delay, contributing to the broader field of computational mathematics and its applications.

For our model problem, we investigate a linear wave-type equation that features a constant time delay, similar to the formulation discussed in \citet{Rodriguez},
\begin{equation}\label{eq_1}
 u_{tt}=c^2u_{xx}+ \lambda u(x,t-\tau)+h(x,t),\  t>\tau, x\in\Omega\subset \mathbb{R}^d  
\end{equation}
with initial conditions
\begin{equation*}
u(x,t)=\phi (x,t), u_t(x,t)=\psi (x,t),\  -\tau\leq t\leq 0, x\in\Omega 
\end{equation*}
and Dirichlet boundary conditions
\begin{equation*}
 u(x,t)=l(t), t\geq 0, x\in\partial\Omega
 \end{equation*}
 where $c$ is the speed at which waves propagate and $\lambda$ is a free parameter. We examine how the DNWR and NNWR methods to the problem (\ref{eq_1}) behave in terms of convergence. Since the equation is linear, we concentrate on the appropriate error equations where $h(x,t)=l(t)=0$ and $\phi(x,t)=0=\psi (x,t)$.





\section{Convergence Analysis of DNWR}
To implement the DNWR method for equation (\ref{eq_1}) in a one-dimensional setting ($d=1$), the spatial domain $\Omega = [-a, b]$ is partitioned into two disjoint subdomains, $\Omega_1 = [-a, 0]$ and $\Omega_2 = [0, b]$. The iterative procedure is initialized with a guess $h^0(t)$ at the interface. Each iteration consists of two sequential steps: first, solving a Dirichlet boundary value problem in $\Omega_1$, followed by, second, solving a Neumann problem in $\Omega_2$. Upon completing an iteration $k$ ($k \in \{1,2, \ldots\})$, the interface condition $h^k(t)$ is updated. The DNWR process can be described as follows: 
 \\
Dirichlet Part:
\begin{equation}
\left\{\begin{array}{rl}\label{eq_2}
\partial_{tt} e_1^k-c^2\partial_{xx} e_1^k-\lambda e_1^k(x, t-\tau ) &=0, \ \ \ (x, t)\in \Omega_1\times (0, T),  \\ 
  e_1^k(x, t)&=0, \  \ \ (x, t)\in \Omega_1\times [-\tau, 0], \\ 
   \partial_t e_1^k(x, t)&=0, \  \ \ (x, t)\in \Omega_1\times [-\tau, 0], \\ 
e_1^k(-a, t)&=0, \ \ \ t\in(0, T),\\ 
e_1^k(0, t)&=h^{k-1}(t), \ \ \  t\in(0, T).
\end{array}\right.
\end{equation}\\
Neumann Part:
\begin{equation}\label{eq_3}
\left\{\begin{array}{rl}
\partial_{tt} e_2^k-c ^2\partial_{xx} e_2^k-\lambda e_2^k(x, t-\tau )&=0, \ (x, t)\in \Omega_2\times (0, T),  \\ 
  e_2^k(x, t)&=0, \  (x, t)\in \Omega_2\times [-\tau, 0], \\ 
  \partial_t e_2^k(x, t)&=0, \  \ \ (x, t)\in \Omega_1\times [-\tau, 0], \\ 
  \partial_x e_2^k(0, t)&=\partial_x e_1^k(0 , t), \  t\in(0, T),\\
e_2^k(b, t)&=0,  \  t\in(0, T).

\end{array}\right.
\end{equation}
For the relaxation parameter $\theta \in (0,1]$, the interface update rule is
\begin{equation}\label{eq_4}
h^{k}( t)=\theta e_2^{k}(0,t) +(1-\theta) h^{k-1}(t).
\end{equation}\\
The primary aim is to examine how the interface error \(h^k(t)\) reduces and approaches zero as \(k \to \infty\), ultimately producing a solution that is smooth throughout the computational domain.
\subsection{Convergence using Fourier Transform}
The convergence is analyzed by applying a Fourier transform in time to the error equations. This converts the time-dependent PDE into a family of parameter-dependent Helmholtz problems in the frequency domain. Analyzing the iteration operator for each Fourier parameter allows one to derive a convergence factor. First, we present the following lemmas.\\
\begin{lemma} \label{lem_1}
    Let $0 < c < d$. The function defined by $g(x) := \frac{\tanh(c x)}{\tanh(d x)}$ is monotonically increasing for $x > 0$ and satisfies the lower bound $g(x) > \frac{c}{d}$.
\end{lemma}
\begin{proof}
    On differentiating $g$ we get,
\begin{equation*}
 g'(x)=\frac{c\tanh(d x)\operatorname{sech}^2(c x)
-d\tanh(c x)\operatorname{sech}^2(d x)}{\tanh^2(d x)}.   
\end{equation*} 
For $g'(x)>0$, we must prove that $c\tanh(d x)\operatorname{sech}^2(c x)>d\tanh(c x)\operatorname{sech}^2(d x)$ which is true since $\tanh(x)$ is an increasing function and $\operatorname{sech}(x)$ is a decreasing function. Additionally, with $x \to 0,$ as the limit, we obtain
       \begin{equation*}
         \frac{\tanh(c x)}{\tanh(d x)} \to \frac{c}{d},
       \end{equation*}
   Hence the Lemma is proved.
\end{proof}
\begin{lemma}\label{lem_2}
    The function $p(x):= \frac{\sinh(c x)}{\sinh(d x)}$ is a monotonically decreasing function for $0<c<d$ and $x>0$ and lies in $\left ( 0,\frac{c}{d} \right )$.\\
\end{lemma} 
\begin{proof}

Given that $\sinh(x)$ is strictly increasing on the interval $(0, \infty)$, the condition $0 < c < d$ implies $\sinh(c x) < \sinh(d x)$. Furthermore, as $x \to \infty$,
\begin{equation*}
\frac{\sinh(c x)}{\sinh(d x)} \to 0,    
\end{equation*}
and when differentiating $p(x)$, we get
\begin{equation*}
  p'(x)=\frac{c\sinh(d x)\cosh(c x)-d\sinh(c x)\cosh(d x)}{\sinh^2(d x)}  .
\end{equation*}
To prove $p'(x)<0$ we need to show $c\sinh(d x)\cosh(c x)<d\sinh(c x)\cosh(d x)$ i.e. $\frac{\tanh(c x)}{\tanh(d x)}>\frac{c}{d}$, which is true by Lemma \ref{lem_1}. Also, as $x \to 0$
\begin{equation*}
  \frac{\sinh(c x)}{\sinh(d x)} \to \frac{c}{d}.  
\end{equation*}
This concludes the proof.
\end{proof}
\begin{lemma}\label{lem_3}
For a complex number $u$, let $\Re(u)$ be its real part and $\Im(u)$ be its imaginary part.
   Consider the function
    \begin{equation*}
        q(u)=\frac{\sinh(cu)}{\sinh(du)},
    \end{equation*} for $0<c<d$.
    Then $q(u)$ satisfies the bound 
    $ \left | q(u)\right |<\frac{c}{d}, \  \forall \ \ u \in \mathbb{D}$ where $
        \mathbb{D}= \left\{ u \in\mathbb{C}|0<\Im(u)<\Re(u)\right\}.$
\end{lemma}

\begin{proof}
     Clearly $q(u)$ is analytic in $\mathbb{D}$, thus the absolute value is given by,
    \begin{equation*}
      \left | q(u)\right |=\frac{\sqrt{\sinh^2(  \left (c \Re (u) \right )+\sin^2(c\Im (u))}}{\sqrt{\sinh^2(d\Re (u))+\sin^2(d\Im(u))}}.  
    \end{equation*}
    Since $\Im(u)<\Re(u)$, so $ |q(u)|<\frac{\sqrt{2}\sinh(c\Re(u))}{\sinh(d\Re(u))}$.
    Therefore, using Lemma \ref{lem_2} we get, $|q(u)|<\sqrt{2}\frac{c}{d}$.\\
    Given that the maximum is attained on the boundary of the domain $\mathbb{D}$, an application of Lemma \ref{lem_2}  yields,
    \begin{equation*}
        \left | q(u)\right |=\max\left\{ \frac{\sinh(c\Re(u))}{\sinh(d\Re(u))},\left | \frac{\sinh(c(1+i)\Re(u))}{\sinh(d(1+i)\Re(u))}\right | \right\}.
    \end{equation*}
    Finally, using Theorem $15.1$ of \citet{complex}, $|q(u)|<\frac{c}{d}, \; \forall u \in \mathbb{D}$.
\end{proof}
\begin{theorem}
For the time-delay hyperbolic PDE (\ref{eq_1}), the DNWR algorithm (\ref{eq_2})--(\ref{eq_4}) with $\theta = 1/2$ satisfies the linear estimates,

for $a > b$:
 \begin{equation*}
    \left\| h^k\right\|_{L^2(\Gamma _T)}\leq \left ( \frac{(a-b)}{2a} \right )^k\left\| h^0\right\|_{L^2(\Gamma_T)}, 
 \end{equation*}
whereas for $a<b$:  
\begin{equation*}
\left\| h^k\right\|_{L^2(\Gamma _T)}\leq \left ( \frac{(b-a)}{2a} \right )^k\left\| h^0\right\|_{L^2(\Gamma_T)},    
\end{equation*}
with $\Gamma _T=\{0\}\times[0,T]$.
 \label{thm:thmdnwr}
\end{theorem}
\begin{proof}
We apply Fourier transform to the error equations (\ref{eq_2})--(\ref{eq_3}) in time, for $k=1,2,3\ldots$, to obtain:\\
Dirichlet Part:
\begin{equation}\label{eq_5}
\left\{\begin{array}{rl}
\frac{\partial^2 \hat{e}_1^k(x, \omega)}{\partial x^2} 
+ \frac{\omega^2+\lambda e^{-i\omega \tau}}{c^2} 
\hat{e}_1^k(x, \omega) &= 0, \\
\hat e_1^k(-a,\omega) &=0,\\
\hat e_1^k(0,\omega) &=\hat h^{k-1}(\omega),
\end{array}\right.
\end{equation}\\
Neumann Part:

\begin{equation}
\label{eq_6}
\left\{\begin{array}{rl}
\frac{\partial^2 \hat{e}_2^k(x, \omega)}{\partial x^2} 
+ \frac{\omega^2 + \lambda e^{-i\omega \tau}}{c^2} 
\hat{e}_2^k(x, \omega) &= 0, \\
\partial_x \hat e_2^k(0, \omega) &=\partial_x \hat e_1^k(0, \omega), \\
\hat e_2^k(b, \omega) &=0, 
\end{array}\right.
\end{equation}\\
the update condition transforms to, 
\begin{equation}\label{eq_7}
\hat h^{k}(\omega)=\theta \hat e_2^{k}(0,\omega) +(1-\theta)\hat h^{k-1}(\omega),
\end{equation}\\
where,
\[
\hat{e}_j^k(x, \omega) = \frac{1}{2\pi} \int_{-\infty}^\infty e_j^k(x, t) e^{-i\omega t} \, dt.
\]
Assuming $\xi ^2=\left ( \frac{\omega^2+\lambda e^{-i\omega \tau}}{c^2} \right )$, we solve the Dirichlet and Neumann boundary value problems to obtain,
\begin{align*}
\hat e_1^k(x,\omega)=\frac{\hat h^{k-1}(\omega)\sinh(\xi (a+x))}{\sinh(\xi a)},
&\quad \hat e_2^k(x,\omega)=-\frac{\hat h^{k-1}(\omega)\coth(\xi a)\sinh(\xi (b-x))}{\cosh(\xi b)},
\end{align*}
and then using the update step we get, \\
\begin{equation*}
 \hat h^k (\omega)=\{ 1-\theta -\theta \coth(\xi a)\tanh(\xi b) \}^k \hat h^0(\omega).   
\end{equation*}
We now focus on the case $\theta=1/2$, that is based on the outcomes in \citet{mandal2025dirichlet} and define 
\begin{equation*}
    Q(\omega)=\tanh(b\xi)\coth(a\xi)-1=\frac{\sinh((b-a)\xi)}{\sinh(a\xi)\cosh(b\xi)},
\end{equation*}
to have \begin{equation*}
  \hat h^k(\omega)=(-1)^k(1/2)^kQ^k(w)\hat h^0(\omega).  
\end{equation*}
\textbf{Case 1:} When $a>b$, we obtain $\left | \frac{\sinh((b-a)\xi)}{\sinh(a\xi)\cosh(b\xi)}\right |< \left | \frac{\sinh((a-b)\xi)}{\sinh(a\xi)}\right |$, as $|\cosh(b\xi)|>1$ for $\Re (\xi)>0$. Therefore, by Lemma \ref{lem_3}, we get $|\frac{\sinh(a-b)\xi}{\sinh(a\xi)}|<\frac{a-b}{a}$.\\~\\
Then, by applying the Parseval-Plancherel identity, we obtain
$$\left\| h^k\right\|_{L^2(\Gamma _T)}\leq \left ( \frac{a-b}{2a} \right )^k\left\| h^0\right\|_{L^2(\Gamma_T)}.$$
\textbf{Case 2:} Let $a<b$. For $b-a<a$ we obtain similarly $\left | \frac{\sinh((b-a)\xi)}{\sinh(a\xi)\cosh(b\xi)}\right |<\frac{b-a}{a}$ using Lemma \ref{lem_3}. In other cases, we use Lemma 2.6 of \citet{gobinda} 
to conclude $\left | \frac{\sinh((b-a)\xi)}{\sinh(a\xi)\cosh(b\xi)}\right |<\frac{b-a}{a}$. Therefore we get,
$$\left\| h^k\right\|_{L^2(\Gamma _T)}\leq \left ( \frac{b-a}{2a} \right )^k\left\| h^0\right\|_{L^2(\Gamma_T)}.$$
Hence the estimates.
\end{proof}
\subsection{Convergence using Laplace Transform}
In this section, we employ the Laplace transform to derive a sharp convergence estimate dependent on the time window size $T$. We first recall the necessary convolution and translation properties, which are subsequently applied to demonstrate the finite termination of the algorithm.\\

\noindent\textbf{Definition (Convolution):}  
The convolution \citep{schiff} of two piecewise continuous functions $f(t)$ and $g(t)$ is defined as:
\begin{equation}\label{convolution}
    (f * g)(t) := \int_{0}^{t} f(\tau) g(t - \tau) \, d\tau.
\end{equation}
Suppose $F(s)$ and $G(s)$ represent the Laplace transforms of the functions $f(t)$ and $g(t)$. The convolution theorem states that the inverse Laplace transform of the product $F(s)G(s)$ is given by the convolution $(f * g)(t)$.\\
\noindent \textbf{Second Translation Theorem:} 
The second translation theorem \citep{schiff} for inverse Laplace transform is given as:
\begin{equation}\label{eq_12}
    \mathcal{L}^{-1} \left( e^{-as} P(s) \right) = H(t-a) p(t - a),
\end{equation}
for  \( a \geq 0 \) and \( P(s) = \mathcal{L}(p(t)) \). The function \( H(t) \) is the Heaviside step function with the form:
\[
H(t) =
\begin{cases}
1, & t \geq 0, \\
0, & t < 0.
\end{cases}
\]

\begin{lemma}\label{lem_efros}
 The inverse Laplace transform of $G(s) =e^{- \alpha\sqrt{s^2 - \lambda e^{-\tau s}}}$ is,
    \begin{equation}\label{inverselaplace}
\mathcal{L}^{-1}\left\{ e^{-\alpha \sqrt{s^2-\lambda e^{-\tau s}}}\right\}= D(\alpha,0,t-\alpha)+\sum_{n\in \mathbb{N}} D(\alpha,\lambda,t-\alpha-n\tau),
\end{equation}
where the term $D(\alpha,\lambda,t-\alpha-n\tau)$ comprises time-shifts and behaves similarly to Heaviside functions.
\end{lemma}
\begin{proof}
  We apply Efros theorem \citep{sana2023dirichlet} which states
if \( \hat{p}(s) \) and \( \hat{r}(s) \exp(-q(s) \xi) \) represent the Laplace transforms of \( p(t) \) and \( r(t, \xi) \) with respect to \( t \), with \( \xi \) being a parameter, then 
\begin{equation*}
    \mathcal{L}^{-1}\left\{\hat{p}(q(s)) \hat{r}(s)\right\}=\int_{0}^{\infty} r(t, \xi) p(\xi) \, d\xi.
\end{equation*}
Assume $r(s)=1$ and $p(s)=e^{-\alpha \sqrt{s}}$ and $q(s)=s^2-\lambda e^{-\tau s}.$
Then we have for $\alpha>0$,
\begin{equation*}
    \mathcal{L}^{-1}\left \{\hat p(q(s)) \hat r(s) \right \} =\mathcal{L}^{-1}\left\{ e^{-\alpha \sqrt{s^2-\lambda e^{-\tau s}}}\right\}=\int_{0}^{\infty}\mathcal{L}^{-1}\left \{ e^{-s^2\xi}e^{\lambda \xi e^{-\tau s}}\right \}\mathcal{L}^{-1}\left \{e^{-\alpha \sqrt{s}}\right \}d\xi.
\end{equation*}
We have by \citet{oberhettinger},
$\mathcal{L}^{-1} \left \{\exp(-\alpha \sqrt{s})\right \}=\frac{\alpha}{2 \sqrt{\pi} t^{3/2}} \exp\left( - \frac{\alpha^2}{4t} \right)$ and $\mathcal{L}^{-1}\left \{e^{-\xi s^2} \right \}=\frac{1}{2\sqrt{\pi\xi}}e^{-t^2/4\xi}.$
Therefore we obtain, 
\begin{equation*}
    \begin{aligned}
&\int_{0}^{\infty}\mathcal{L}^{-1}\left \{ e^{-s^2\xi}e^{\lambda \xi e^{-\tau s}}\right \}\mathcal{L}^{-1}\left \{e^{-\alpha \sqrt{s} }\right \}d\xi =\\
&\int_{0}^{\infty}\mathcal{L}^{-1}\left \{ e^{-s^2\xi}+\lambda \xi e^{-s^2\xi} e^{-\tau s} +\frac{(\lambda \xi)^2 e^{-s^2\xi}e^{-2\tau s}}{2!}+ \cdots \right \}\frac{\alpha}{2 \sqrt{\pi} \xi^{3/2}} \exp\left( - \frac{\alpha^2}{4\xi} \right)d\xi \\
&=\frac{\alpha}{2 \sqrt{\pi}}\int_{0}^{\infty}\mathcal{L}^{-1}\left \{ e^{-s^2\xi}\right\}\frac{1}{\xi^{3/2}} \exp\left( - \frac{\alpha^2}{4\xi} \right)d\xi + \sum_{n\in \mathbb{N}} D(\alpha,\lambda,t-\alpha-n\tau),
    \end{aligned}
\end{equation*}
where the first term corresponds to the case $\lambda=0$ or the classical wave equation case without the delay term. Now for $\lambda=0$, we will simply obtain the term $\mathcal{L}^{-1} \left \{\exp(-\alpha s)\right \}\hat h^0(s)$, as studied in \citet{dd22}. This first term thus produces a Heaviside function because of the second translation theorem \eqref{eq_12}. Thus the other terms in $D(\alpha,\lambda,t-\alpha-n\tau)$ also comprise consecutive time-shifts due to the presence of the term $\exp(-n\tau s)$, as illustrated in Fig \ref{timeshift}. Therefore, we obtain, 
\begin{equation}\label{inverselaplace}
\mathcal{L}^{-1}\left\{ e^{-\alpha \sqrt{s^2-\lambda e^{-\tau s}}}\right\}= D(\alpha,0,t-\alpha)+\sum_{n\in \mathbb{N}} D(\alpha,\lambda,t-\alpha-n\tau).
\end{equation}
\end{proof}
\begin{theorem}\label{thm_2}
    Let $a$ and $b$ denote the widths of the two subdomains and $c$ be the wave speed. Applying the DNWR algorithm to the time-delayed wave PDE \eqref{eq_2}--\eqref{eq_4} with $\theta = 1/2$ guarantees convergence within $k+1$ iterations, subject to the condition that the time window $T$ satisfies:
    \[
        \frac{T}{k} \leq 2 \min\left\{ \frac{a}{c}, \frac{b}{c} \right\}.
    \]
\end{theorem}
\begin{proof}
By taking the Laplace transform of the subproblems \eqref{eq_2}--\eqref{eq_3} and of the interface update condition \eqref{eq_4}, we derive the following solutions on subdomain $\Omega_1,\Omega_2$:
\begin{align*}
\hat e_1^k(x,s) &=\frac{\hat h^{k-1}(s)\times \sinh 
 \left(\frac{\sqrt{s^2-\lambda e^{-\tau s}}(x+a)}{c }\right)}{ \sinh\left(\frac{a\sqrt{s^2-\lambda e^{-\tau s}}}{c} \right)},
\\
\\
\hat e_2^k(x,s) &=\frac{\hat h^{k-1}(s)\times \coth\left(\frac{a\sqrt{s^2-\lambda e^{-\tau s}}}{c} \right)\times \sinh \left(\frac{\sqrt{s^2-\lambda e^{-\tau s}}(b-x)}{c}\right)}{ \cosh\left(\frac{b\sqrt{s^2-\lambda e^{-\tau s}}}{c}\right)},
\end{align*}
now by induction, we derive the following update condition:
\begin{equation*}
    \hat h^k(s)=\left(1-\theta-\theta \tanh \left(\frac{b\sqrt{s^2-\lambda e^{-\tau s}}}{c}\right)\times  \coth\left(\frac{a\sqrt{s^2-\lambda e^{-\tau s}}}{c}\right)\right)^k\hat h^0(s).
\end{equation*}
We rewrite the update step as,
\begin{equation}\label{eq_11}
   \hat h^k(s)=\left(1-2\theta- \theta K^a_b(s) \right)^k\hat h^0(s),  
\end{equation}
where we assume the form of the kernel, \begin{equation*}
  K^a_b(s)=\left( \coth\left(\frac{a\sqrt{s^2-\lambda e^{-\tau s}}}{c}\right) \times \tanh \left(\frac{b\sqrt{s^2-\lambda e^{-\tau s}}}{c}\right)-1\right).   
\end{equation*}
For the ease of simplification, we assume $X=\frac{\sqrt{s^2-\lambda e^{-\tau s}}}{c}$ and then expanding the hyperbolic functions in terms of exponential series yields,
\begin{equation*}
    K^a_b(s)=2\sum_{m=1}^{\infty }e^{-2amX}-2\sum_{n=1}^{\infty }(-1)^{n-1}e^{-2nbX}-4\sum_{m=1}^{\infty }\sum_{n=1}^{\infty }(-1)^{n-1}e^{-2(nb+ma)X}.
\end{equation*}
So, we obtain from equation \eqref{eq_11}, replacing $\theta=1/2$,
\begin{align*}
    \hat{h}^k(s) &= \left(-\frac{1}{2}\right)^k \left(K^a_b(s)\right)^k \hat{h}^0(s)= \left(-1\right)^k \bigg\{ \sum_{m=1}^{\infty } e^{-2amX} 
    - \sum_{n=1}^{\infty } (-1)^{n-1} e^{-2nbX} \notag \\
    &\quad - 2 \sum_{m=1}^{\infty } \sum_{n=1}^{\infty } (-1)^{n-1} e^{-2(nb+ma)X} \bigg\}^k\hat{h}^0(s).
\end{align*}
Thus, we simplify
\begin{align*}
\hat h^k(s) = \Bigg[ & -e^{-2aX} + e^{-2bX}
                      + \Bigg( \sum_{m>1}^{\infty} e^{-2amX} 
                              - \sum_{n>1}^{\infty} (-1)^{n-1} e^{-2nbX} - 2\sum_{m=1}^{\infty} \sum_{n=1}^{\infty} (-1)^{n-1} 
                              e^{-2(nb+ma)X} \Bigg) \Bigg]^k \; \hat h^0(s),
\end{align*}
further using binomial expansion, we obtain,
\begin{align*}
\hat h^k(s) = & (-1)^k e^{-2akX} \hat h^0(s) 
                 + e^{-2bkX} \hat h^0(s)
                 + \Bigg( \sum_{l>k}^{\infty } P_l^{(k)} e^{-2blX} + \sum_{l>k}^{\infty } Q_l^{(k)} e^{-2alX}
                  + \sum_{\substack{m+n \geq k \\ m,n \geq 1}}^{\infty } 
                          R_{m,n}^{(k)} e^{-2(nb+ma)X} \Bigg) \hat h^0(s),
\end{align*}
therefore, the inverse Laplace of above equation yields
\begin{equation}
    h^k(t)=(-1)^k \mathcal{L}^{-1}\left\{ e^{-2akX}\hat h^0(s)\right \}+\mathcal{L}^{-1}\left\{ e^{-2bkX}\hat h^0(s)\right\}+\mathcal{L}^{-1}\left\{ \text{other terms}\right\}.
    \label{eq11a}
\end{equation}
Now, using the convolution theorem (\ref{convolution}), we have:
$\mathcal{L}^{-1}\{G(s) \hat h^0(s)\} = (g * h^0)(t),$ where,
\begin{equation}\label{inversecrucial}
   G(s) = e^{- \alpha\sqrt{s^2 - \lambda e^{-\tau s}}}\ \ \text{and} \ \ F(s)=\hat h^0(s). 
\end{equation}
We now calculate the expression of the inverse Laplace of $G(s) =e^{- \alpha\sqrt{s^2 - \lambda e^{-\tau s}}}$ using Lemma \ref{lem_efros} which gives,

\begin{equation}\label{inverselaplace}
\mathcal{L}^{-1}\left\{ e^{-\alpha \sqrt{s^2-\lambda e^{-\tau s}}}\right\}= D(\alpha,0,t-\alpha)+\sum_{n\in \mathbb{N}} D(\alpha,\lambda,t-\alpha-n\tau).
\end{equation}
In particular, the presence of the delay $\tau$ introduces terms of the form 
$\sum_{n\in \mathbb{N}} D(\alpha,\lambda,t-\alpha-n\tau)$, which represent a sequence of time-shifted contributions. This structure reflects a delay-induced memory effect, where the solution depends on past states and leads to a progressive shift in time.
Therefore, we obtain from \eqref{eq11a}, 
\begin{align*}
h^k(t)&=(-1)^k a_1H(t-2ak/c)h^0(t-2ak/c)+a_2H(t-2bk/c)h^0(t-2bk/c)\\
&+(-1)^ka_3H(t-2ak/c-\tau)h^0(t-2ak/c-\tau)+a_4H(t-2bk/c-\tau)h^0(t-2bk/c-\tau)\\
&+\sum_{\ell>k}\sum_{n\geq0}\left[p_{\ell,n}^{(k)}H(t-n\tau-2al/c)h^0(t-n\tau-2a\ell/c)\right.\\
&\left.+q_{\ell,n}^{(k)}H(t-n\tau-2b\ell/c)h^0(t-n\tau-2b\ell/c)\right]\\
&+\sum_{m+\nu\geq k}\sum_{n\geq0}r_{m,\nu,n}^{(k)}H(t-n\tau-2(am+b\nu)/c)h^0(t-n\tau-2(am+b\nu)/c).
\end{align*}
Here $a_i, p_{\ell,n}^{(k)}, q_{\ell,n}^{(k)}, r_{m,\nu,n}^{(k)}$ are real constants. Because the first term in our analysis contains no delay, it behaves similarly to a standard wave equation and error becomes identically zero when $T\leq \alpha$. Similarly, the error associated with the summation terms vanishes when $T \leq \alpha + n\tau$. \\Now if one chooses the time window \(T\) such that
\[
\frac{T}{k} \leq 2 \min\left\{\frac{a}{c}, \frac{b}{c}\right\},
\]
then $h^k(t)\equiv 0$ in $[0,T]$. Therefore, the DNWR algorithm provides the desired solution in another iteration. This completes the result. 
\end{proof}

\begin{remark}
    Another way to evaluate the inverse Laplace of the expression $e^{-\alpha \sqrt{s^2 - \lambda e^{-\tau s}}} \hat h^0(s)$ in \eqref{inversecrucial} is the following:\\
Let $G(s)$ be the kernel function and $g(t)$ its inverse transform:
$G(s) = e^{-\alpha \sqrt{s^2 - \lambda e^{-\tau s}}}$ with 
\begin{equation*}
    y(t) = (g * h^0)(t) = \int_0^t g(\xi) h^0(t-\xi) d\xi.
\end{equation*}
We then use the series expansion for the modified Bessel function of the first kind $I_1(z)$ and the identity for $\mathcal{L}^{-1}\{e^{-\alpha \sqrt{s^2 - k^2}}\}$. The parameter $k^2$ is replaced by the term $\lambda e^{-\tau s}$, that depends on $s$; we treat it as an operator or expand the expression using the series definition of $I_1(z)$:
\begin{equation*}
    I_1(z) = \sum_{m=0}^{\infty} \frac{(z/2)^{2m+1}}{m!(m+1)!}
\end{equation*}
Substituting $k = \sqrt{\lambda} e^{-\frac{\tau}{2} s}$ into the formal expansion of the transform:
\begin{equation*}
   G(s) = e^{-\alpha s} + \sum_{m=0}^{\infty} \frac{\alpha (\sqrt{\lambda} e^{-\frac{\tau}{2} s})^{2m+2}}{2^{2m+1} m!(m+1)!} \mathcal{L}\left\{ (t^2 - \alpha^2)^m u(t-\alpha) \right\}. 
\end{equation*}
Simplifying the powers of the exponential and the constants:
\begin{equation*}
    G(s) = e^{-\alpha s} + \sum_{m=0}^{\infty} \frac{\alpha \lambda^{m+1} e^{-(m+1)\tau s}}{2^{2m+1} m!(m+1)!} \int_{\alpha}^{\infty} e^{-st} (t^2 - \alpha^2)^m dt.
\end{equation*}
Applying \eqref{eq_12} to each term in the summation, the kernel $g(t)$ is derived as:
\begin{equation*}
    g(t) = \delta(t-\alpha) + \sum_{n=1}^{\infty} \frac{\alpha \lambda^n}{2^{2n-1} n! (n-1)!} \left[ (t - n\tau)^2 - \alpha^2 \right]^{n-1} H(t - \alpha - n\tau).
\end{equation*}
Finally, applying the convolution integral, the first term (the Dirac delta) shifts the function $h^0(t)$, and the summation terms involve shifted integrals:
\begin{equation*}
\begin{aligned}
&\mathcal{L}^{-1}\left\{ e^{-\alpha \sqrt{s^2-\lambda e^{-\tau s}}}\hat h^0(s)\right\} = h^0(t-\alpha)H(t-\alpha)\\
&+ \sum_{n=1}^{\infty} \frac{\alpha \lambda^n}{2^{2n-1} n! (n-1)!}
\left[ \int_{\alpha+n\tau}^t \left((\xi-n\tau)^2 - \alpha^2\right)^{n-1} h^0(t-\xi) d\xi \right] H(t - \alpha - n\tau).
\end{aligned}    
\end{equation*}
\end{remark}
\begin{remark}
  We illustrate the effect of time shifting property of the factor $e^{-2 \sqrt{s^2 - e^{-3s}}}$ through inverse Laplace transform in Fig.~\ref{timeshift}. 
The exponential factor $e^{-2\phi(s)}, \quad \text{with} \quad \phi(s)=\sqrt{s^2 - e^{-3s}},
$ induces a delay effect analogous to the classical time-shifting property of Laplace transforms. In particular, it implies that the resulting function is effectively shifted by approximately $2$ units in time. This shows that the contribution of the error is vanishing for times $t<2$. Consequently, if one restricts the analysis to a time window $(0,T)$ with $T<2$, the error does not enter the computational domain and is therefore effectively zero. This shifting property is the primary reason for the method's rapid convergence. It provides the theoretical justification for why Theorem \ref{thm_2} will successfully resolve in a definite number of iterations when the time window is appropriately bounded. 
Thus, the time-shifting property plays a central role: it ensures that, for sufficiently small time windows, the error is effectively “pushed out” of the domain after a 'finite' number of iterations, which explains the finite-step convergence behaviour.
\end{remark}
\begin{remark}
   The convergence behaviour of the DNWR method obtained using Fourier and Laplace framework can be interpreted differently:
    \begin{itemize}
        \item Fourier analysis approach, based on Parseval-Plancherel identities, yields a contraction factor strictly less than one, leading to \textbf{a linear convergence} estimate to the numerical error.
        \item In contrast, Laplace transform analysis reveals additional information through the time-shifting property. In particular, exponential factors in Laplace space correspond to delays in the time domain. As a consequence, the error is transported forward in time at each iteration. For sufficiently small time intervals, this implies that after a finite number of iterations, the error is completely shifted outside the time window, resulting in \textbf{finite-step convergence}.
    \end{itemize}

\end{remark}
    \begin{figure}
    \centering
    \includegraphics[width=0.7 \linewidth]
    {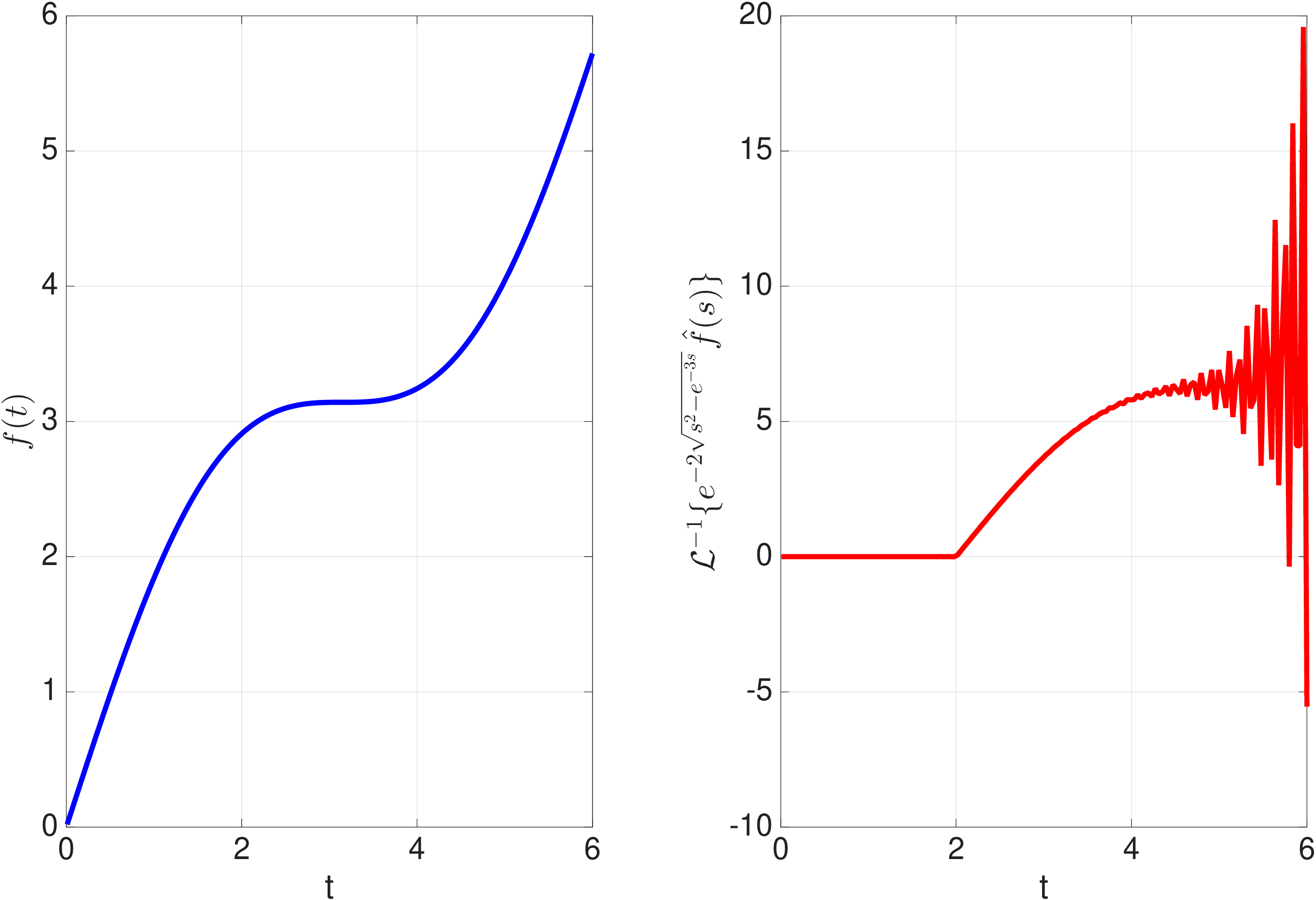}
     \caption{Illustration of the time-shifting effect on the function $f(t) = t+\sin(t)$. The left figure displays the original inverse transform $\mathcal{L}^{-1}\{ \hat f(s)\}$, while the right figure shows the transformed signal $\mathcal{L}^{-1}\{e^{-2 \sqrt{s^2 - e^{-3s}}}\hat f(s)\}$.}
     \label{timeshift}
    \end{figure}

\section{Convergence of DNWR for Heterogeneous Media}
We study the problem with a piecewise-constant wave speed, where \( c(x,t) = c_1 \) in region \( \Omega_1 \) and \( c(x,t) = c_2 \) in region \( \Omega_2 \). The error equations corresponding to the DNWR algorithm are as follows:\\
\textbf{Dirichlet Part:}
\begin{equation*}
 \left\{\begin{array}{rl}
\partial_{tt} e_1^k-c_1^2 \partial_{xx} ^2e_1-\lambda e_1(x, t-\tau ) &=0, \ \ \ (x, t)\in \Omega_1\times (0, T),  \\ 
   e_1^k(x, t)&=0, \  \ \ (x, t)\in \Omega_1\times [-\tau, 0], \\ 
    \partial_te_1^k(x, t)&=0,\ \ \  (x, t)\in \Omega_1\times[-\tau,0], \\
  e_1^k(-a, t)&=0, \ \ \ t\in(0, T),\\ 
e_1^k(\Gamma, t)&=h^{k-1}(t), \ \ \  t\in(0, T).
\end{array}\right.    
\end{equation*}
\textbf{Neumann Part:}
\begin{equation*}
  \left\{\begin{array}{rl}
\partial_{tt} e_2^k-c_2^2 \partial_{xx} ^2e_2-\lambda e_2(x, t-\tau )&=0, \ (x, t)\in \Omega_2\times (0, T),  \\ 
  e_2^k(x, t)&=0, \  (x, t)\in \Omega_2\times [-\tau, 0], \\ 
   \partial_te_2^k(x, t)&=0,\ \  (x, t)\in \Omega_2\times[-\tau,0], \\
e_2^k(b, t)&=0, \ \ \  t\in[0, T],\\ 
c_2^2\partial_x e_2^k(\Gamma , t)&=c_1^2\partial_x e_1^k(\Gamma , t), \  t\in(0, T).
\end{array}\right.  
\end{equation*}

The interface update condition is:
\begin{equation*}
h^{k}(x, t)=\theta e_2^{k}\mid _{\Gamma\times(0, T)} +(1-\theta) h^{k-1}(x, t).    
\end{equation*}
The DNWR method applied to the heterogeneous wave equation with time delay, under the specific geometric condition where the scaled subdomain sizes satisfy \( |\Omega_1|/{c_1} = |\Omega_2|/{c_2} \), exhibits a particularly efficient convergence rate. The result is the following:
\begin{theorem}
In DNWR for the heterogeneous wave equation with time delay, satisfying $|\Omega_1|/{c_1}=|\Omega_2|/{c_2}$, the method exhibits 2 step convergence for $\theta=1/(1+\frac{c_1}{c_2})$, regardless of the length of the time window.
 \label{thm:thmdnwrhetero}
 \end{theorem}
 \begin{proof}
The proof proceeds by taking the Laplace transform of the iterative DNWR error equations. For iteration index \(k = 1, 2, 3, \dots\), the transformed error functions \(\hat e_1^k(x,s)\) and \(\hat e_2^k(x,s)\) are obtained as solutions to the resulting boundary value problems, 
\begin{align*}
\hat e_1^k(x,s) &= \hat h^{k-1}(s)\, \frac{\sinh\left(\frac{\sqrt{s^2-\lambda e^{-\tau s}}(x+a)}{c_1}\right)}{\sinh\left(\frac{a\sqrt{s^2-\lambda e^{-\tau s}}}{c_1}\right)}, \\
\hat e_2^k(x,s) &= -\hat h^{k-1}(s)\, \frac{{c_1}}{{c_2}}\, \coth\left(\frac{a\sqrt{s^2-\lambda e^{-\tau s}}}{c_1}\right) \frac{\sinh\left(\frac{\sqrt{s^2-\lambda e^{-\tau s}}(b-x)}{c_2}\right)}{\cosh\left(\frac{a\sqrt{s^2-\lambda e^{-\tau s}}}{c_1}\right)}.
\end{align*}
The relaxation update in the Laplace domain is
\begin{equation*}
 \hat h^k(s) = \theta\times\hat e_2^k(0,s) + (1-\theta)\times\hat h^{k-1}(s).   
\end{equation*}
Substituting the expression for \(\hat e_2^k(0,s)\) and simplifying yields the recursive relation
\begin{equation*}
\hat h^k(s) = \Bigl(1-\theta - \theta{\frac{c_1}{c_2}}\; \coth\!\Bigl(\tfrac{a\sqrt{s^2-\lambda e^{-\tau s}}}{c_1}\Bigr)
\tanh\!\Bigl(\tfrac{b\sqrt{s^2-\lambda e^{-\tau s}}}{c_2}\Bigr)\Bigr)^{\!k}\, \hat h^0(s).    
\end{equation*}
Introducing the shorthand $A = a/{c_1}$ and $B = b/{c_2}$ and defining
\begin{equation*}
\hat f(s) = \coth\!\bigl(A\sqrt{s^2-\lambda e^{-\tau s}}\bigr)\,\tanh\!\bigl(B\sqrt{s^2-\lambda e^{-\tau s}}\bigr) - 1,    
\end{equation*}
we may rewrite the factor as
\begin{equation*}
\hat h^k(s) = \Bigl(\bigl[1-\bigl(1+{\tfrac{c_1}{c_2}}\bigr)\theta\bigr] - {\tfrac{c_1}{c_2}}\,\theta\,\hat f(s)\Bigr)^{\!k}\,\hat h^0(s).    
\end{equation*}
When the scaled subdomain lengths are equal, i.e., \(A = B\) (equivalently \(|\Omega_1|/{c_1}=|\Omega_2|/{c_2}\)), the function \(\hat f(s)\) vanishes. In that case the inverse Laplace transform gives
\begin{equation*}
h^k(t) = \Bigl[1-\bigl(1+{\tfrac{c_1}{c_2}}\bigr)\theta\Bigr]^{\,k}\, h^0(t).    
\end{equation*}
After choosing the relaxation parameter as \(\theta = 1/\bigl(1+{c_1/c_2}\bigr)\), the factor becomes zero, so that \(h^k(t) = 0\) for \(k \ge 1\). Hence the DNWR iteration converges exactly within two iterations under this parameter choice and geometric condition.
 \end{proof}

 \section{Multi-Subdomain Convergence Analysis of NNWR}
We examine the convergence of NNWR for the one-dimensional wave equation with time delay on the domain $\Omega = (0, D)$. 
By imposing continuity criteria at the interfaces, the spatial domain $\Omega$ is divided into $N$ disjoint subdomains $\Omega_i = (x_{i-1}, x_i)$.  The length of the $i$-th subdomain is represented by $h_i = x_i - x_{i-1}$, while the minimal subdomain size is represented by $h_{\min} = \min_{1 \leq i \leq N} h_i$.
The initial guess for the solution at the interfaces $x_i$ is represented by $\{ g^0_i(t) \}_{i=1}^{N-1}$. We assume homogeneous initial and boundary conditions due to the linearity of the problem.  
For all $k$, we set $g_0^k(t) = g_N^k(t) = 0$, allowing us to reformulate the problem accordingly. First, we solve dirichlet problems on $\Omega_i$ in parallel.
\begin{equation}\label{eq_13}
\left\{\begin{array}{rl}
\partial_{tt} e_i^k-c ^2\partial_{xx} e_i^k-\lambda e_i^k(x, t-\tau )&=0, \ (x, t)\in \Omega_i\times (0, T),  \\ 
  e_i^k(x, t)&=0, \  (x, t)\in \Omega_i\times [-\tau, 0], \\ 
  \partial_te_i^k(x, t)&=0, \ \ \ (x, t)\in \Omega_i\times[-\tau,0], \\ 
e_i^k(x_{i-1} , t)&=g^{k-1}_{i-1}(t) \  t\in(0, T),\\
 e_i^k(x_i,t)&= g_i^{k-1}(t), \  t\in(0, T),
\end{array}\right.
\end{equation}\\
followed by the correction step for all $k$.
\begin{equation}\label{eq_14}
\left\{\begin{array}{rl}
\partial_{tt} \phi_i^k-c ^2\partial_{xx} \phi_i^k-\lambda \phi_i^k(x, t-\tau )&=0, \ (x, t)\in \Omega_i\times (0, T),  \\ 
  \phi_i^k(x, t)&=0, \  (x, t)\in \Omega_i\times [-\tau, 0], \\
   \partial_t\phi_i^k(x, t)&=0,\ \ \  (x, t)\in \Omega_i\times[-\tau,0], \\
-\partial_x \phi_i^k(x_{i-1},t)&= (\partial_x e^k_{i-1}-\partial_x e_i^k)(x_{i-1},t), \  t\in(0, T), \\
\partial_x \phi_i^k(x_{i},t)&= (\partial_x e^k_{i}-\partial_x e_{i+1}^k)(x_{i},t)\  t\in(0, T).
\end{array}\right.
\end{equation}\\
The interface update condition, with $\theta$ denoting the relaxation parameter, is defined as:
\begin{equation}
g_i^{k}(t) = g_i^{k-1}( t)-(\theta\times(\phi_i^k(x_i,t)+\phi_{i+1}^k(x_i,t))).\\
\end{equation}
The goal is to ensure that the solution in each subdomain matches at the interfaces.

\begin{theorem}{(NNWR convergence for multiple subdomains)}\label{theorem4}
For the wave PDE with time delay, the NNWR algorithm \eqref{eq_13}-\eqref{eq_14} achieves convergence within at most $k+1$ iterations for $\theta=1/4$, assuming the time window $T$ adheres to the constraint $T/k \leq 2h_{min}/c$, where $c$ represents the wave propagation speed.
\end{theorem}
\begin{proof}
\textbf{Proof Outline:} The proof proceeds in the following main steps:
\begin{enumerate}
    \item[i.] 
    We begin by applying the Laplace transform to the homogeneous Dirichlet and Neumann subproblems to obtain their local analytical solutions in the Laplace space.
    \item[ii.] 
    We incorporate these local solutions into the update condition and apply mathematical induction with $\theta = 1/4$. This allows us to express the interface error at the $k$-th iteration solely in terms of the initial error at $k=0$, with coefficients $\hat r_{i,j}$ that involve hyperbolic functions.
    \item[iii.] 
    We expand the coefficients into infinite series of exponential terms.
    \item[iv.] 
    Finally, we apply the inverse Laplace transform. Using the Efros Theorem 
    we demonstrate that for a sufficiently small time window $T \leq 2k h_{\min}/c$, the error terms become zero, thereby achieving exact convergence.
\end{enumerate}
    On applying the Laplace transform to the homogeneous Dirichlet problems we get,
    \begin{equation}
        s^2 \hat e_i^k-c^2\hat  e_{i,xx}^k-\lambda e^{-s\tau}\hat e_i^k =0,\ \ \hat e_i^k(x_{i-1} , s)=\hat g^{k-1}_{i-1}(s),\ \ \hat e_i^k(x_i,s)= \hat g_i^{k-1}(s).
    \end{equation}   
Let $\beta_i=\cosh(h_i\sqrt{s^2-\lambda e^{-\tau s}}/c)$ and $\alpha_i=\sinh(h_i\sqrt{s^2-\lambda e^{-\tau s}}/c)$ for $i=2,3,\ldots N-1.$ Then the solution within the subdomain is given by:
\begin{equation}
    \hat e_i^k=\frac{1}{\alpha_i }\left ( \hat g_i^{k-1}(s)\sinh\left ( \frac{\sqrt{s^2-\lambda e^{-\tau s}}}{c}(x-x_{i-1}) \right )+ \hat g_{i-1}^{k-1}(s)\sinh\left ( \frac{\sqrt{s^2-\lambda e^{-\tau s}}}{c}(x_{i}-x) \right ) \right).
\end{equation}
The solution to Neumann subproblems on applying Laplace transform are :
\begin{equation}
    \hat \phi_i^k=\left ( A_i(s)\cosh\left ( \frac{\sqrt{s^2-\lambda e^{-\tau s}}}{c}(x-x_{i-1}) \right )+ B_{i}(s)\cosh\left ( \frac{\sqrt{s^2-\lambda e^{-\tau s}}}{c} (x_{i}-x)\right ) \right),
\end{equation}
where, \begin{equation*}
    A_i=\frac{1}{\alpha_i }\left (\hat g_i^{k-1}\left ( \frac{\beta_i}{\alpha_i} +\frac{\beta_{i+1}}{\alpha_{i+1}}\right )-\frac{\hat g_{i-1}^{k-1}}{\alpha_i}-\frac{\hat g_{i+1}^{k-1}}{\alpha_{i+1}} \right ),
\end{equation*} 
\begin{equation*}
B_i=\frac{1}{\alpha_i }\left (\hat g_{i-1}^{k-1}\left ( \frac{\beta_i}{\alpha_i} +\frac{\beta_{i-1}}{\alpha_{i-1}}\right )-\frac{\hat g_{i-2}^{k-1}}{\alpha_{i-1}}-\frac{\hat g_{i}^{k}}{\alpha_{i}} \right ).    
\end{equation*} 
By induction from update step, $i=2,3,\ldots N-2$ we get,
\begin{align*}
g^{k}_i (s) = \hat{g}^{k-1}_i (s) - \theta 
\left( \hat{\phi}^{k}_i (x_i , s) + \hat{\phi}^{k}_{i+1} (x_i , s) \right) 
\end{align*}
\begin{align*}
\Rightarrow g^{k}_i = \hat{g}^{k-1}_i (s) - \theta (A_i \beta_i + B_i + A_{i+1} + B_{i+1} \beta_{i+1}).
\end{align*}
Using the identity \( \beta_i^2 - 1 = \alpha_i^2 \) and simplifying, we obtain
\begin{align}
\hat{g}^{k}_i = \hat{g}^{k-1}_i - \theta \Bigg( &\hat{g}^{k-1}_i \left(2 + \frac{2\beta_i \beta_{i+1}}{\alpha_i \alpha_{i+1}}\right)+ \frac{\hat{g}^{k-1}_{i+1}}{\alpha_{i+1}} \left( \frac{\beta_{i+2}}{\alpha_{i+2}} - \frac{\beta_i}{\alpha_i} \right) \notag \\
&+ \frac{\hat{g}^{k-1}_{i-1}}{\alpha_i} \left( \frac{\beta_{i-1}}{\alpha_{i-1}} - \frac{\beta_{i+1}}{\alpha_{i+1}} \right) -\frac{\hat{g}^{k-1}_{i+2}}{\alpha_{i+1} \alpha_{i+2}}
- \frac{\hat{g}^{k-1}_{i-2}}{\alpha_i \alpha_{i-1}} \Bigg).
\label{eq:update}
\end{align}
After applying the Laplace transform, the solution for the first and last subdomains, which have homogeneous Dirichlet physical boundaries, is given by:
\begin{align*}
\hat{\phi_1}(x,s) &= \frac{1}{\beta_1} \left(\hat g_1 \left( \frac{\beta_1}{\alpha_1} + \frac{\beta_2}{\alpha_2} \right )- \frac{\hat{g_2}}{\alpha_2} \right) \sinh \left( \frac{(x - x_0)\sqrt {s^2-\lambda e^{-\tau s}}}{c} \right), \\
\hat{\phi}_N (x,s) &= \frac{1}{\beta_N} \left( \hat g_{N-1}\left (\frac{\beta_{N-1}}{ \alpha_{N-1}} + \frac{\beta_N}{\alpha_N}\right ) - \frac{\hat{g}_{N-2}}{\alpha_{N-1}} \right) \sinh \left( \frac{(x_N - x)\sqrt{s^2-\lambda e^{-\tau s}}}{c} \right).
\end{align*}
The update conditions on the first and the last interfaces are
\begin{align}
\hat{g}^{k}_1 = \hat{g}^{k-1}_1 - \theta \Bigg(&\hat{g}^{k-1}_1 \left(2 + \frac{\beta_1 \beta_2}{\alpha_1 \alpha_2} + \frac{\alpha_1 \beta_2}{\beta_1 \alpha_2} \right) +\frac{\hat{g}^{k-1}_2}{\alpha_2} \left( \frac{\beta_3}{\alpha_3} - \frac{\alpha_1}{\beta_1} \right) -\frac{\hat{g}^{k-1}_3}{\alpha_2 \alpha_3} \Bigg),
\end{align}
\begin{align}
\hat{g}^{k}_{N-1} = \hat{g}^{k-1}_{N-1} - \theta \Bigg(&\hat{g}^{k-1}_{N-1} \left(2 + \frac{\beta_{N-1} \beta_N}{\alpha_{N-1} \alpha_N} + \frac{\alpha_N \beta_{N-1}}{\beta_N \alpha_{N-1}} \right)+ \frac{\hat{g}^{k-1}_{N-2}}{\alpha_{N-1}} \left( \frac{\beta_{N-2}}{\alpha_{N-2}} - \frac{\alpha_N}{\beta_N} \right) \notag \\
&- \frac{\hat{g}^{k-1}_{N-3}}{\alpha_{N-1} \alpha_{N-2}} \Bigg).
\label{eq:22}
\end{align}
From the update step (\ref{eq:update}) for $\theta=1/4$, we get
\begin{align}
\hat{g}^{k}_i (s) = -\frac{1}{4} \Big( & \hat{r}_{i,i} \hat{g}^{k-1}_i (s) + \hat{r}_{i,i+1} \hat{g}^{k-1}_{i+1} (s) + \hat{r}_{i,i-1} \hat{g}^{k-1}_{i-1} (s) 
- \hat{r}_{i,i+2} \hat{g}^{k-1}_{i+2} (s) - \hat{r}_{i,i-2} \hat{g}^{k-1}_{i-2} (s) \Big),
\label{eq:23}
\end{align}
where we denote:
\begin{align*}
 \hat {r}_{i,i} = \frac{2}{\alpha_i \alpha_{i+1}} (\beta_i \beta_{i+1} - \alpha_i \alpha_{i+1}),
 \hat{r}_{i,i+1} = \frac{(\alpha_i \beta_{i+2} - \beta_i \alpha_{i+2})}{\alpha_i \alpha_{i+1} \alpha_{i+2}},
 \hat{r}_{i,i-1} &= \frac{(\alpha_{i+1} \beta_{i-1} - \beta_{i+1} \alpha_{i-1})}{ \alpha_{i-1} \alpha_i\alpha_{i+1}},
 \end{align*}
 
$$ \hat{r}_{i,i+2} = \frac{1}{\alpha_{i+1} \alpha_{i+2}},
 \hat{r}_{i,i-2} = \frac{1}{\alpha_i \alpha_{i-1}}.$$\\
Similarly we can write $\hat g^k_1(s)$ and $\hat g^k_{N-1}(s)$,
\begin{align}
\hat{g}^{k}_1(s) &= -\frac{1}{4} \Big( \hat{r}_{1,1} \hat{g}^{k-1}_1(s) + \hat{r}_{1,2} \hat{g}^{k-1}_2(s) - \hat{r}_{1,3} \hat{g}^{k-1}_3(s) \Big), 
\label{eq:24}
\end{align}
\begin{align*}
\text{where}\ \ \hat{r}_{1,1} &= \left( \frac{\alpha_1\beta_2}{\beta_1 \alpha_2} + \frac{\beta_{1}\beta_{2}}{\alpha_1 \alpha_2} - 2 \right),\
\hat{r}_{1,2} = \frac{1}{\alpha_2} \left( \frac{\beta_3}{\alpha_3} - \frac{\alpha_1}{\beta_1} \right),\
\hat{r}_{1,3} = \frac{1}{\alpha_2 \alpha_3}.
\end{align*}
\begin{align}
\hat{g}^{k}_{N-1}(s) &= -\frac{1}{4} \Big( \hat{r}_{N-1,N-1} \hat{g}^{k-1}_{N-1}(s) + \hat{r}_{N-1,N-2} \hat{g}^{k-1}_{N-2}(s) - \hat{r}_{N-1,N-3} \hat{g}^{k-1}_{N-3}(s) \Big),
\label{eq:25}
\end{align}
\begin{align*}
\text{where}\ \ \hat{r}_{N-1,N-1} &= \left( \frac{\alpha_{N-1}\beta_{N-2}}{\beta_{N-1} \alpha_{N-2}} + \frac{\beta_{N-1}\beta_{N-2}}{\alpha_{N-1} \alpha_{N-2}} - 2 \right),\\
\hat{r}_{N-1,N-2} &= \frac{1}{\alpha_{N-2}} \left( \frac{\beta_{N-3}}{\alpha_{N-3}} - \frac{\alpha_{N-1}}{\beta_{N-1}} \right),\
\hat{r}_{N-1,N-3} = \frac{1}{\alpha_{N-2} \alpha_{N-3}}.
\end{align*}
Note that $\hat{r}_{i,i+1} = -\hat{r}_{i+1,i}$, $\hat{r}_{i,i+2} = \hat{r}_{i+2,i} $. Thus, by applying induction on the equations (\ref{eq:23})-(\ref{eq:25}), we obtain:
\begin{equation}
\hat{g}_i^{k} (s) =
\sum_{j=-2n}^{2n}
\left( -\frac{1}{4} \right)^n
q_{i+j}^{n}
\left(
\hat{r}_{i+j,i+j-2}, \hat{r}_{i+j,i+j-1}, \dots, \hat{r}_{i,i}, \dots, \hat{r}_{i+j,i+j+1}, \hat{r}_{i+j,i+j+2}
\right)
\hat{g}_{i+j}^{k-n} (s),
\label{eq:26}
\end{equation}
and

\begin{equation}
\hat{g}_1^{k} (s) =
\sum_{j=0}^{2n}
\left( -\frac{1}{4} \right)^n
q_{1+j}^{n}
\left(
\hat{r}_{1,1}, \dots, \hat{r}_{1+j,2+j}, \hat{r}_{1+j,3+j}
\right)
\hat{g}_{1+j}^{k-n} (s).
\label{eq:27}
\end{equation}
Here, the coefficients $q^n_{i+j}$ and $q^n_{1+j}$ denote homogeneous polynomials of degree $n$.
Similarly we can write the expression for \( \hat{g}_{N-1}^{k}(s) \). Now denoting by \( X=\frac{\sqrt{s^2-\lambda e^{-\tau s}}}{c}\) and using the binomial series expansion for the hyperbolic functions, we derive the following:
\begin{align*}
\hat{r}_{i,i} &= 2 \frac{\cosh \left( (h_i - h_{i+1}) X \right)}{\sinh(h_i X) \sinh(h_{i+1} X)} = 4 ( e^{-2h_i X} + e^{-2h_{i+1} X}) \\
&\quad \times \left[ 1+ \sum_{m=1}^{\infty} e^{-2h_i m X} + \sum_{n=1}^{\infty} e^{-2h_{i+1} n X}+ \sum_{m=1}^{\infty} \sum_{n=1}^{\infty} e^{-2(m h_i + n h_{i+1}) X} \right],
\end{align*}
\begin{align*}
    \hat r_{i,i+1} &= \frac{\sinh \left((h_i - h_{i+2})X \right)}{\sinh\left(h_i X \right) \sinh\left(h_{i+1} X \right) \sinh\left(h_{i+2} X \right)} \\
    &= 4 \Bigg[1 + \sum_{l=1}^{\infty} e^{-2l h_i X} + \sum_{m=1}^{\infty} e^{-2m h_{i+1} X} + \sum_{n=1}^{\infty} e^{-2n h_{i+2} X} \\
    & \quad + \sum_{m=1}^{\infty} \sum_{n=1}^{\infty} \left(e^{-2(m h_i + n h_{i+1})X} + e^{-2(m h_{i+1} + n h_{i+2})X} + e^{-2(m h_{i+2} + n h_{i})X} \right) \\
    & \quad + \sum_{l=1}^{\infty} \sum_{m=1}^{\infty} \sum_{n=1}^{\infty} e^{-2(l h_i + m h_{i+1} + n h_{i+2})X} \Bigg] \left(e^{-(h_{i+1} + 2 h_{i+2})X} - e^{-(h_{i+1} + 2 h_i)X} \right),
\end{align*}

\begin{align*}
\hat{r}_{i,i+2} &= 
\frac{1}{\sinh(h_{i+1} X) \sinh(h_{i+2} X)} = 4e^{-(h_{i+1} + h_{i+2}) X}\notag\\
&\Bigg[
1 + \sum_{m=1}^{\infty} e^{-2m h_{i+1} X}
+\sum_{n=1}^{\infty} e^{-2n h_{i+2} X} +\sum_{m=1}^{\infty} \sum_{n=1}^{\infty} e^{-2(m h_{i+1} + n h_{i+2}) X}
\Bigg],
\end{align*}

\begin{align*}
\hat{r}_{1,1} &= 
\frac{2 \cosh((2h_1 - h_2) X)}{\sinh(2h_1 X) \sinh(h_2 X)}= 4 \left( e^{-4h_1 X} + e^{-2h_2 X} \right) \notag \\
&\Bigg[
1 + \sum_{m=1}^{\infty} e^{-4m h_1 X}
+ \sum_{n=1}^{\infty} e^{-2n h_2 X}  + \sum_{m=1}^{\infty} \sum_{n=1}^{\infty} e^{-2(2m h_1 + n h_2) X}
\Bigg],
\end{align*}

\begin{align*}
\hat{r}_{1,2} &= 
\frac{\cosh((h_1 - h_3) X)}{\cosh(h_1 X) \sinh(h_2 X) \sinh(h_3 X)}=\notag \\ & 4 
\Bigg[
1+  \sum_{l=1}^{\infty} (-1)^l e^{-2l h_1 X}
+ \sum_{m=1}^{\infty} e^{-2m h_2 X} 
+ \sum_{n=1}^{\infty} e^{-2n h_3 X} \notag \\
&\quad + \sum_{m=1}^{\infty} \sum_{n=1}^{\infty} 
\Big( (-1)^m e^{-2(m h_1 + n h_2) X} 
+ (-1)^m e^{-2(m h_1 + n h_3) X} 
+ e^{-2(m h_2 + n h_3) X} \Big) \notag \\
&\quad + \sum_{l=1}^{\infty} \sum_{m=1}^{\infty} \sum_{n=1}^{\infty} 
(-1)^l e^{-2(l h_1 + m h_2 + n h_3) X} 
\Bigg] \times \Big(e^{-(2h_1 + h_2) X} + e^{-(h_2 + 2h_3) X} \Big).
\end{align*}
Similarly, we can write for the terms $\hat{r}_{N-1,N-3}$, $\hat{r}_{N-1,N-2}$, $\hat{r}_{N-1,N-1}$, $\hat{r}_{i,i-2}$,  $\hat{r}_{i,i-1}$ and $\hat{r}_{1,3}$ .
Now using the above expressions, equations (\ref{eq:26})-(\ref{eq:27}) become
\begin{equation}
\hat{g}_i^k (s) = (-1)^k \left[ \left( e^{-2kh_iX} + e^{-2kh_{i+1}X} \right) \hat{g}_i^0 (s) + \sum_{j=-2k}^{2k} t_{i+j}^{k}(s) \hat{g}_{i+j}^0 (s) \right],
\label{eq:28}
\end{equation}
and
\begin{equation}
\hat{g}_1^k (s) = (-1)^k \left[ \left( e^{-4h_1kX} + e^{-2h_2kX} \right) \hat{g}_1^0 (s) + \sum_{j=0}^{2k} p_{1+j}^{k}(s) \hat{g}_{1+j}^0 (s) \right].
\label{eq:29}
\end{equation}
Here $t_{i+j}^{k}(s)$ and $p_{1+j}^{k}(s)$ are linear combinations of exponential terms $e^{-\alpha X}$, where $\alpha = 2k h_l$ for indices $l \in \{1, 2, \dots, N\}$. Similarly, we can write for $\hat{g}_{N-1}^k (s)$. Utilizing \eqref{eq_12} to compute the inverse transform of \eqref{eq:28}--\eqref{eq:29}, we derive results similar to the wave equation without delay \citep{wavemulti} i.e. when $T \leq 2k h_{\min}/c$, we get $g_i^k (t) = 0$ for all $i$, because from (\ref{inverselaplace}) we have 

$\begin{aligned}
    \mathcal{L}^{-1} \left\{ e^{ -\frac{2k h_{\min}}{c} \sqrt{s^2 - \lambda e^{-\tau s}} } \right\}
&=
D\left(\frac{2k h_{\min}}{c}, 0, t - \frac{2k h_{\min}}{c}\right)\\
&+ 
\sum_{n \in \mathbb{N}} D\left(\frac{2k h_{\min}}{c}, \lambda, t - \frac{2k h_{\min}}{c} - n\tau \right).
\end{aligned}$\\
Therefore, we get 
\begin{align*}
g_i^k(t)&=(-1)^k a_iH(t-2kh_{i}/c)g_i^0(t-2kh_i/c)+a_{i+1}H(t-2kh_{i+1}/c)g_i^0(t-2kh_{i+1}/c)\\
&+(-1)^kb_iH(t-2kh_{i}/c-\tau)g_i^0(t-2kh_i/c-\tau)\\
&+b_{i+1}H(t-2kh_{i+1}/c-\tau)g_i^0(t-2kh_{i+1}/c-\tau)+\text{other terms},
\end{align*}
and
\begin{align*}
g_1^k(t)&=(-1)^k a_1H(t-2kh_{1}/c)g_1^0(t-2kh_1/c)+a_2H(t-2kh_{2}/c)g_1^0(t-2kh_{2}/c)\\
&+(-1)^kb_1H(t-2kh_{1}/c-\tau)g_1^0(t-2kh_1/c-\tau)\\
&+b_2H(t-2kh_{2}/c-\tau)g_1^0(t-2kh_{2}/c-\tau)+ \text {other terms}.
\end{align*}
Here $a_i, b_i$ are real constants. Similarly, we can write expression for $g^k_{N-1}(t)$ and hence the algorithm converges when $T\leq 2kh_{min}/c$.
\end{proof}
\section{NNWR Convergence Analysis in 2-D}
This section details the formulation and convergence analysis of the NNWR algorithm for the two-dimensional wave equation with a time delay. To assess the performance of the NNWR method on Equation~(\ref{eq_1}) in two spatial dimensions, we analyze the corresponding error equation under homogeneous boundary conditions in both the \(x\)- direction as well as \(y\)-direction, along with a zero-valued history function.

\begin{figure}
    \centering
    \includegraphics[width=0.8 \linewidth]{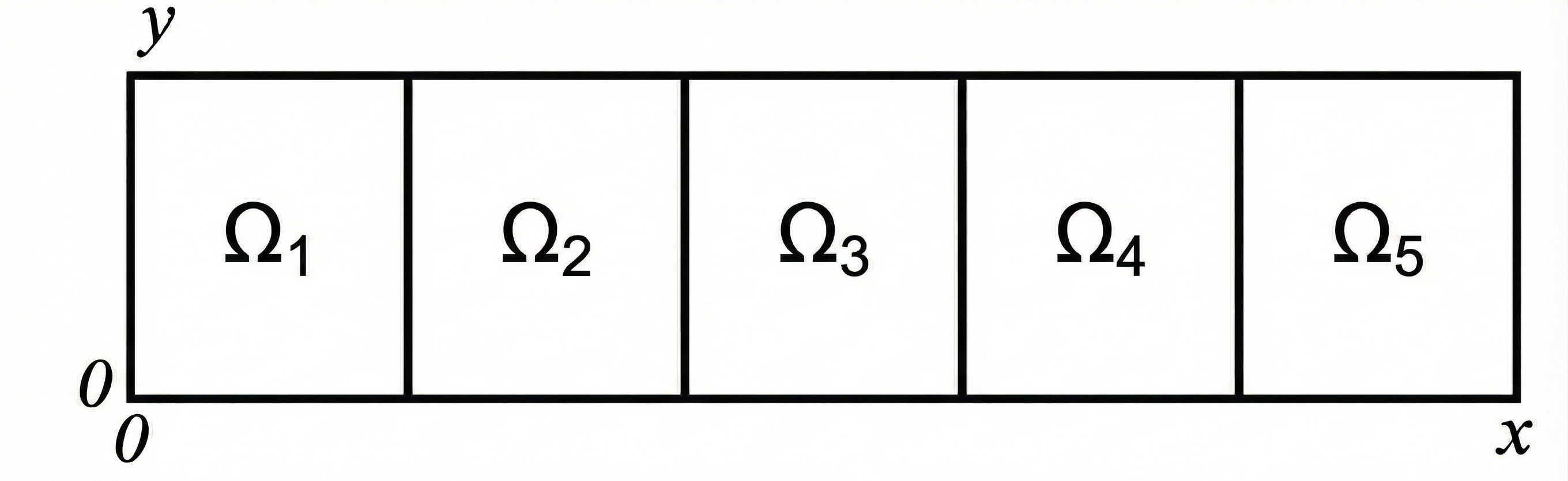}
   \caption{Illustration of 2D spatial domain decomposition into strips}
     \label{2Dstrips}
    \end{figure}
In the formulation of the NNWR algorithm, the spatial domain $\Omega$ is partitioned into a collection of vertical strips (Fig. \ref{2Dstrips}), given by
    
\[
\Omega_i = (x_{i-1}, x_i) \times (0, \pi), \quad \text{for } i = 1, \ldots, N,
\]
where \( x_0 = l < x_1 < \cdots < x_N = L \). The width of each subdomain is represented by \( h_i := x_i - x_{i-1} \), and we define the minimal subdomain width as
$
h_{\min} := \min_{1 \leq i \leq N} h_i.$

Starting with initial interface approximations $\{\chi^0_i(y, t)\}_{i=1}^{N-1}$ at $\{x = x_i\}$, the NNWR algorithm is iteratively applied in the form of alternating Dirichlet and Neumann steps. Specifically, for each iteration index \( k = 1, 2, \ldots \) and for each subdomain index \( i = 1, \ldots, N \), the algorithm performs the following:
\begin{itemize}
    \item a Dirichlet solve using the interface values from iteration \( k-1 \),
    \item after which the interface values are updated via a Neumann correction.
\end{itemize}
This iterative procedure corresponds to a two-dimensional extension of the NNWR system as expressed in the equations (\ref{eq_13}-\ref{eq_14}).
\begin{equation}\label{nnwr_2d}
\left\{\begin{array}{rl}
\partial_{tt} e_i^k-c ^2(\partial_{xx} e_i^k+\partial_{yy}e_i^k)-\lambda e_i^k(x,y, t-\tau )&=0, \ (x,y, t)\in \Omega_i\times (0, T),  \\ 
  e_i^k(x, y,t)&=0, \  (x,y, t)\in \Omega_i\times [-\tau, 0], \\ 
  \partial_te_i^k(x,y, t)&=0, \ \ \ (x,y, t)\in \Omega_i\times[-\tau,0], \\ 
e_i^k(x_{i-1} ,y, t)&=g^{k-1}_{i-1}(y,t) \  t\in(0, T),\\
 e_i^k(x_i,y,t)&= g_i^{k-1}(y,t), \  t\in(0, T),\\
 e_i^k(x,0,t)&=e_i^k(x,\pi,t)=0,t\in(0,T),
\end{array}\right.
\end{equation}\\
Subsequently, the correction step is applied for all $k$.
\begin{equation}
\left\{\begin{array}{rl}\label{nnwr_2d_cor}
\partial_{tt} \phi_i^k-c ^2(\partial_{xx} \phi_i^k+\partial_{yy}\phi_i^k)-\lambda \phi_i^k(x, y,t-\tau )&=0, \ (x,y,t)\in \Omega_i\times (0, T),  \\ 
  \phi_i^k(x, y,t)&=0, \  (x,y,t)\in \Omega_i\times [-\tau, 0], \\
   \partial_t\phi_i^k(x,y, t)&=0,\ \ \  (x,y,t)\in \Omega_i\times[-\tau,0], \\
-\partial_x \phi_i^k(x_{i-1},y,t)&= (\partial_x e^k_{i-1}-\partial_x e_i^k)(x_{i-1},y,t), \  t\in(0, T), \\
\partial_x \phi_i^k(x_{i},y,t)&= (\partial_x e^k_{i}-\partial_x e_{i+1}^k)(x_{i},y,t)\  t\in(0, T),\\
\phi_i^k(x,0,t)&=\phi_i^k(x,\pi,t)=0,t\in(0,T),
\end{array}\right.
\end{equation}
The interface update condition is:
\begin{equation}
\chi_i^{k}(y,t) = \chi_i^{k-1}(y, t)-\theta(\phi_i^k(x_i,y,t)+\phi_{i+1}^k(x_i,y,t)).\\
\end{equation}
To simplify the problem, we apply a Fourier sine transform in the \( y \)-direction, thereby reducing the original two-dimensional problem into a set of one-dimensional problems. Thus, the solution is represented by a Fourier sine series as
\[
e^k_i(x, y, t) = \sum_{n=1}^{\infty} E^k_i(x, n, t) \sin(ny),
\]
with the coefficients defined as
\[
E^k_i(x, n, t) = \frac{2}{\pi} \int_0^{\pi} e^k_i(x, \eta, t) \sin(n \eta)\, d\eta.
\]
As a result, we are decomposing the 2-D problem into many independent 1-D problems for each mode $n$. 
\begin{equation}\label{2d_to_1d}
\frac{\partial^2 E^k_i}{\partial t^2}(x, n, t) - c^2 \frac{\partial^2 E^k_i}{\partial x^2}(x, n, t) + c^2 n^2 E^k_i(x, n, t)-\lambda E_i^k(x,n, t-\tau ) = 0, 
\end{equation}
with the appropriate boundary conditions for \( E^k_i(x, n, t) \). Before we present the final convergence result, the below auxiliary result is required.

\begin{lemma}
We have the following identity using Bessel functions of the first kind $J_\nu(z)$,
\begin{equation*}
  \begin{aligned}
\mathcal{L}^{-1}\left\{ e^{-\gamma \sqrt{s^2+\beta^2-\lambda e^{-\tau s}}}\right\} &= \left[ \delta(t - \gamma) - \frac{\beta \gamma J_1(\beta \sqrt{t^2 - \gamma^2})}{\sqrt{t^2 - \gamma^2}} \right] H(t - \gamma) \\
&+ \frac{\gamma \lambda}{2} J_0(\beta \sqrt{(t - \tau)^2 - \gamma^2}) H(t - \tau - \gamma) \\
&+ \frac{\gamma \lambda^2}{8} \left[ \frac{t - 2\tau - \gamma}{\beta(t - 2\tau + \gamma)} J_1(\beta \sqrt{(t - 2\tau)^2 - \gamma^2})\right. \\
&+ \left.\frac{\gamma}{2} J_0(\beta \sqrt{(t - 2\tau)^2 - \gamma^2}) \right] H(t - 2\tau - \gamma) + \dots
\end{aligned}  
\end{equation*}
\end{lemma}

\begin{proof}
To find the inverse Laplace transform
\begin{equation}\label{newsi}
    \Psi(\gamma,\beta,\lambda,\tau,t) := \mathcal{L}^{-1}\left\{e^{-\gamma \sqrt{s^2 + \beta^2 - \lambda e^{-\tau s}}}\right\},
\end{equation}
we rewrite the function as:
\begin{equation*}
    \hat \Psi(s) = \exp\left( -\gamma \sqrt{s^2 + \beta^2} \sqrt{1 - \frac{\lambda e^{-\tau s}}{s^2 + \beta^2}} \right).
\end{equation*}
Using the binomial expansion of $(1-x)^{1/2}$, with $x = \frac{\lambda e^{-\tau s}}{s^2 + \beta^2}$, we have:
\begin{equation*}
   \hat \Psi(s) = e^{-\gamma \sqrt{s^2 + \beta^2}} \cdot \exp\left( -\gamma \sum_{k=1}^{\infty} \binom{1/2}{k} \frac{(-\lambda)^k e^{-k \tau s}}{(s^2 + \beta^2)^{k - 1/2}} \right), 
\end{equation*}
which is further simplified using the Taylor series expansion of the 2nd exponential term:
\begin{equation*}
    \hat \Psi(s) = e^{-\gamma \sqrt{s^2 + \beta^2}} \sum_{m=0}^{\infty} \frac{(-\gamma)^m}{m!} \left( \sum_{k=1}^{\infty} \binom{1/2}{k} \frac{(-\lambda)^k e^{-k \tau s}}{(s^2 + \beta^2)^{k - 1/2}} \right)^m.
\end{equation*}
The term for $m=0$ is simply $e^{-\gamma \sqrt{s^2 + \beta^2}}$. For $m \ge 1$, the expression involves powers of a series. We can collect terms by the total power of $e^{-\tau s}$. Let $N$ be the total index such that the term contains $e^{-N \tau s}$. This results in a series of the form:
\begin{equation*}
  \hat \Psi(s) = e^{-\gamma \sqrt{s^2 + \beta^2}} + \sum_{N=1}^{\infty} C_N \frac{e^{-\gamma \sqrt{s^2 + \beta^2}} e^{-N \tau s}}{(s^2 + \beta^2)^{\nu_N/2}}, 
\end{equation*}
where $C_N$ are coefficients and $\nu_N$ are corresponding powers.

We utilize the known inverse Laplace transform identities involving Bessel functions of the first kind $J_\nu(z)$.
From \citet{wavemulti}, we have,
\begin{equation*}
    \mathcal{L}^{-1} \left\{ e^{-\gamma \sqrt{s^2 + \beta^2}} \right\} = \delta(t - \gamma) - \frac{\beta\gamma}{\sqrt{t^2 - \gamma^2}} J_1(\beta\sqrt{t^2 - \gamma^2}) H(t - \gamma).
\end{equation*}
For the general term, we use results from \citet{oberhettinger,schiff} to obtain:
\begin{equation*}
    \mathcal{L}^{-1} \left\{ \frac{e^{-\gamma \sqrt{s^2 + \beta^2}}}{(s^2 + \beta^2)^{\nu/2}} \right\} = \left( \frac{t - \gamma}{t + \gamma} \right)^{\frac{\nu-1}{2}} \frac{1}{\beta^{\nu-1}} J_{\nu-1}(\beta \sqrt{t^2 - \gamma^2}) H(t - \gamma).
\end{equation*}
The factor $e^{-N \tau s}$ introduces further time shift due to \eqref{eq_12}. Combining the terms, the inverse transform $\Psi(\gamma,\beta,\lambda,\tau,t)$ is expressed as a sum of delayed signals:
\begin{equation*}
\begin{aligned}
\Psi(\gamma,\beta,\lambda,\tau,t) &= \left[ \delta(t - \gamma) - \frac{\beta \gamma}{\sqrt{t^2 - \gamma^2}} J_1(\beta \sqrt{t^2 - \gamma^2}) \right] H(t - \gamma) \\
&+ \frac{\gamma \lambda}{2} \left( \frac{t - \tau - \gamma}{t - \tau + \gamma} \right)^{0} J_0(\beta \sqrt{(t - \tau)^2 - \gamma^2}) H(t - \tau - \gamma) \\
&+ \sum_{N=2}^{\infty} \mathcal{L}^{-1} \left\{ \text{Higher order terms in } \lambda \right\}.
\end{aligned}    
\end{equation*}
The general term for the $N$-th order contribution in $\lambda$ (considering only the linear contribution from the $m=1$ expansion for simplicity in notation) is:
\begin{equation*}
  f_N(t) = -\gamma \binom{1/2}{N} (-\lambda)^N \left( \frac{t - N\tau - \gamma}{t - N\tau + \gamma} \right)^{N - 1} \frac{1}{\beta^{2N-2}} J_{2N-2}(\beta \sqrt{(t - N\tau)^2 - \gamma^2}) H(t - N\tau - \gamma).  
\end{equation*}
The complete solution is the sum of these contributions:
\begin{equation*}
    \Psi(\gamma,\beta,\lambda,\tau,t) = \left[ \delta(t - \gamma) - \frac{\beta \gamma J_1(\beta \sqrt{t^2 - \gamma^2})}{\sqrt{t^2 - \gamma^2}} \right] H(t - \gamma) + \sum_{N=1}^{\infty} A_N(t) H(t - N\tau - \gamma),
\end{equation*}
where $A_N(t)$ represents the inverse transform of the $N$-th order term in the expansion of the exponential function. This completes the result.
\end{proof}
%

\begin{theorem}[NNWR Convergence in 2D]
Let the relaxation parameter \( \theta = 1/4 \). For a fixed time window length \( T \in (0,\infty)\), the NNWR algorithm (\ref{nnwr_2d})-(\ref{nnwr_2d_cor}) achieves convergence in at most \( k + 1 \) iterations, provided
$$
T/k < 2 h_{\min}/c,
$$
where \( c \) denotes the wave propagation speed.
\end{theorem}\label{thm5}
\begin{proof}
Applying the Laplace transform in time to equation (\ref{2d_to_1d})  yields
\[
(s^2 + c^2 n^2-\lambda e^{-\tau s}) \, \hat{E}^k_i(x, s) - c^2 \frac{d^2 \hat{E}^k_i}{dx^2} = 0.
\]
Now for each mode $n$, we treat the system just like we did in 1-D in Theorem \ref{theorem4}, where the recurrence relations for the interface functions had the form:
\begin{equation*}
\hat{g}^k_i(s) = \sum_j A^{(k)}_{ij}\left( \sqrt{s^2-\lambda e^{-\tau s}} \right) \ \, \hat{g}^0_j(s).    
\end{equation*}
In 2D the update conditions will be modified for each Fourier mode \( n = 1, 2, \ldots \), as follows:
\begin{equation}\label{updatestep2d}
    \hat \chi^k_i(n, s) = \sum_j A^{(k)}_{ij}\left( \sqrt{s^2 + c^2 n^2-\lambda e^{-\tau s}} \right) \, \hat\chi^0_j(n, s).
\end{equation}
So, again, we express the interface values in the Laplace domain as a linear combination of initial data, modulated by these coefficients $A_{ij}^k\left( \sqrt{s^2 + c^2 n^2-\lambda e^{-\tau s}} \right)$.
 The coefficients $A^{(k)}_{ij}\left( \sqrt{s^2-\lambda e^{-\tau s}} \right)$ in 1-D are linear combinations of exponential terms of the form $e^{-\alpha X}$, where $X= \frac{\sqrt{s^2-\lambda e^{-\tau s}}}{c}$. Using the Efros theorem and exponential series expansion, we find that the presence of a $\lambda$-term causes additional delay, so we only take into account the first term of the series without terms containing $\lambda$. The modified coefficients \( A^{(k)}_{ij}(\sqrt{s^2 + c^2 n^2-\lambda e^{-\tau s}}) \) become sums of exponential terms of the form \( e^{-\gamma \sqrt{s^2 + c^2 n^2 -\lambda e^{-\tau s}}} \), where $\gamma\geq 2kh_l/c$. Therefore, we use the
definition of $\Psi(\gamma,\beta,\lambda,\tau,t)$ in \eqref{newsi} to compute the inverse Laplace transform of \eqref{updatestep2d}, 
\begin{equation*}
  \chi_{i}^{k}(n,t)=\sum_{j}\sum_{l}\Psi(\rho_{i,l,j,k},cn,\lambda,\tau,t)*\chi_{j}^{0}(n,t),  
\end{equation*}
with $\rho_{i,l,j,k}\geq2kh_{\min}/c$. So for $t<2kh_{\min}/c$, $\chi_{i}^{k}(n,t)=0$ for each $n$. So, when \( t < \frac{2k h_{\min}}{c} \), after one more iteration, all interface discrepancies are eliminated and the NNWR algorithm recovers the exact solution across the entire spatial domain.
\end{proof}
\section{Numerical illustrations}
For the numerical experiments, we consider the error equations corresponding to the model problem~\ref{eq_1}. 
The wave equation is discretized using the \textit{Leapfrog scheme}, employing a centered finite difference scheme 
in both space and time. We discretize the time domain such that the delay $\tau$ is an integer multiple of the time step $\Delta t$, i.e., $\tau = m \Delta t$. The resulting scheme is given by
\[
u_i^{n+1}
=
2u_i^n - u_i^{n-1}
+ r^2 \left( u_{i+1}^n - 2u_i^n + u_{i-1}^n \right)
+ \Delta t^2 \lambda\, \big(u_i^{n-m}\big),
\]
where, $
r = \frac{c\,\Delta t}{\Delta x}
$ satisfies the CFL condition $r \leq 1 $. The parameters are chosen as $c=1$ and $\lambda=1.6$ for all the experiments. The initial guess is set as $h^0(t) = t^2$ for $t \in (0,T]$ at all interface boundaries. The delay parameter is taken as $\tau=3$ unless otherwise specified.
\subsection{DNWR asymmetric convergence}
For asymmetric decomposition we consider two subdomains $\Omega_1 = (0,4)$ and $\Omega_2 = (4,6)$, so that $a = 4$ and $b = 2$, 
implying $a > b$ and vice versa. For discretization we choose mesh size $\Delta x = 0.025=\Delta t$ which is the time step. 
 In Fig.~\ref{fourier_estimate}, we compare the theoretically obtained bound via Fourier 
analysis with the numerical convergence result of the DNWR method for $\theta = 1/2$. The numerical results agree with the theoretical estimate, thereby validating the bound established in 
Theorem~\ref{thm:thmdnwr}.

Fig.~\ref{diftime_wave} illustrates the convergence behavior for various time window lengths $T$, confirming the results established in Theorem \ref{thm_2}. Specifically, the DNWR method converges in at most $k+1$ iterations provided that $\frac{T}{k} \leq 2 \min\left\{ \frac{a}{c}, \frac{b}{c} \right\}$. We evaluate this using two experimental settings: \textbf{Case I:} with $\Omega_1 = (0,4.5) \ \& \ \Omega_2 = (4.5,6)$, and \textbf{Case II:} with $\Omega_1 = (0,2.5) \ \& \ \Omega_2 = (2.5,6)$.  For example, in Case I, with $c=1$ and $\min\{a,b\}/c=1.5$, choosing $T=3$ yields
$T/2\min\left\{ a/c, b/c \right\} = 1.$
This implies $k=1$, demonstrating that the DNWR method converges in at most $1+1=2$ iterations, as verified in the left panel of Fig.~\ref{diftime_wave}. 
It is also worth noting that the magnitude (or strength) of the time delay does not directly affect the convergence behavior of the proposed algorithms. Instead, convergence is primarily governed by the length of the time window over which the solution is computed.
Furthermore, Fig.~\ref{difdelay} presents results for different values of the delay parameter $\tau$ in the DNWR algorithm for the two-subdomain case. The results confirm that variations in $\tau$-value do not influence the convergence behavior.  

\begin{figure}
    \centering
    \includegraphics[width=0.462\linewidth]{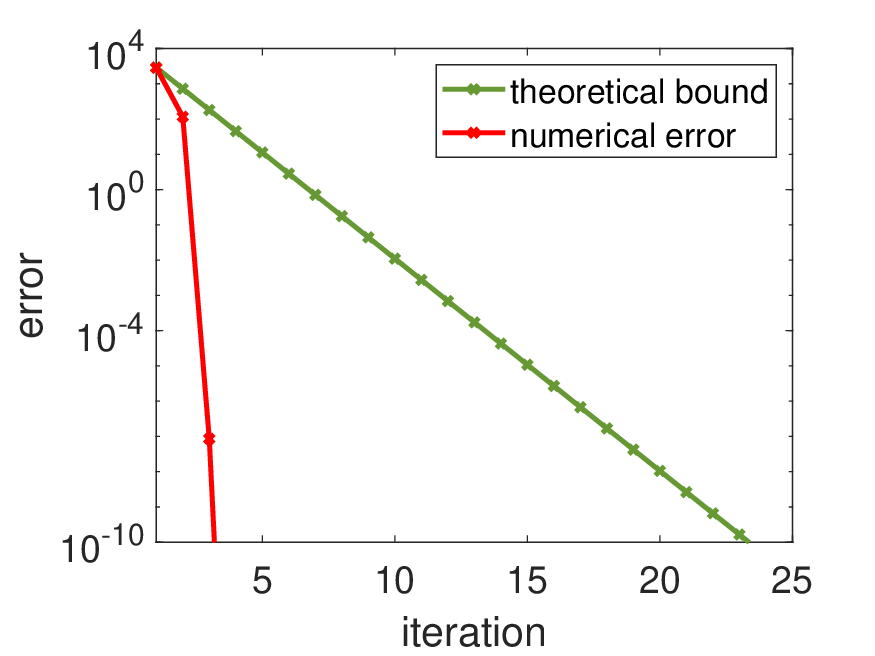}
    \includegraphics[width=0.462\linewidth]{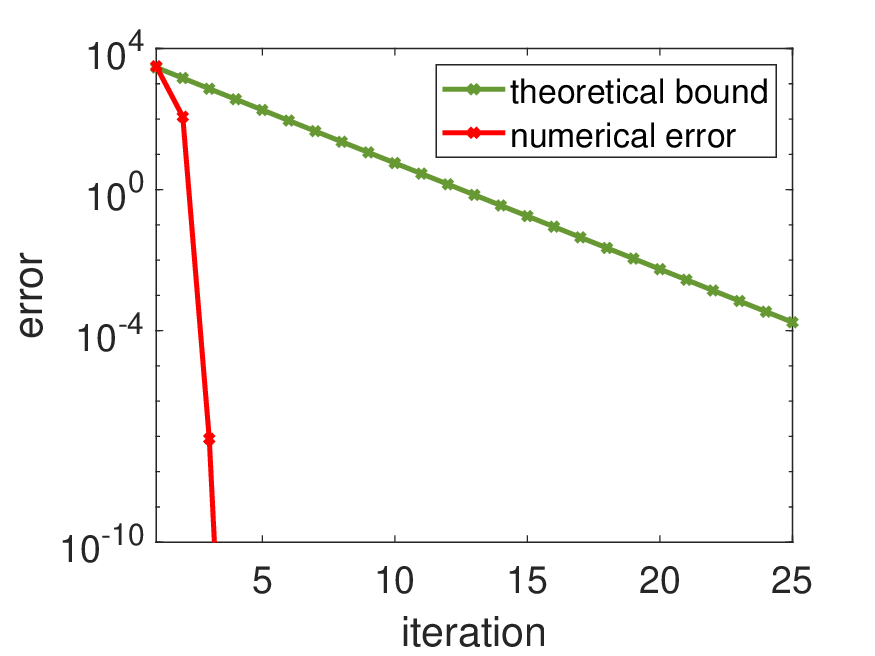}
   
   \caption{Comparison of the DNWR numerical convergence against the theoretical bound for $\theta=1/2$. Figure left corresponds to the case $a>b$, while the right one depicts the case $b>a$}
    \label{fourier_estimate}
\end{figure}

\begin{figure}
    \centering
    \includegraphics[width=0.462\linewidth]{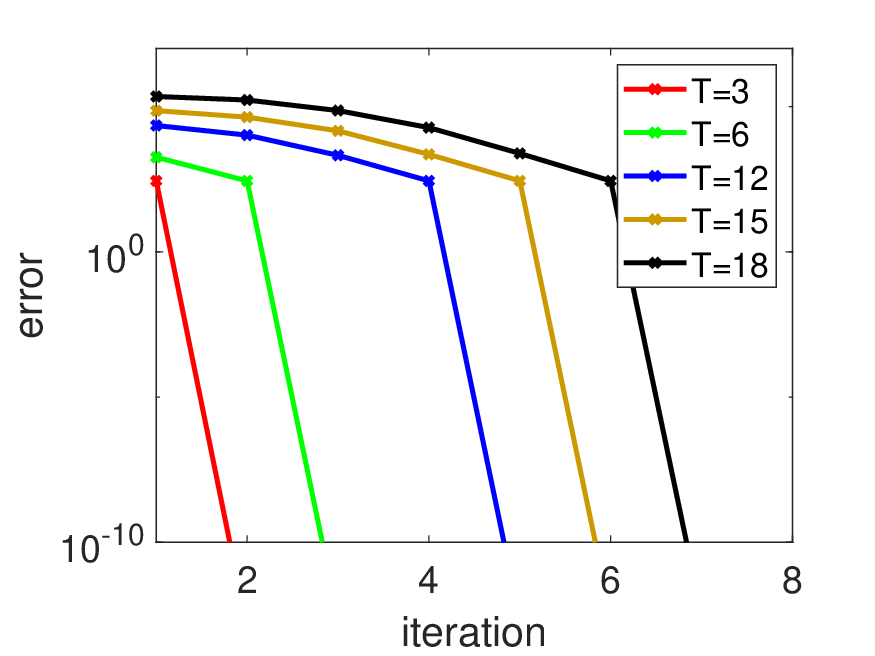}
    \includegraphics[width=0.462\linewidth]{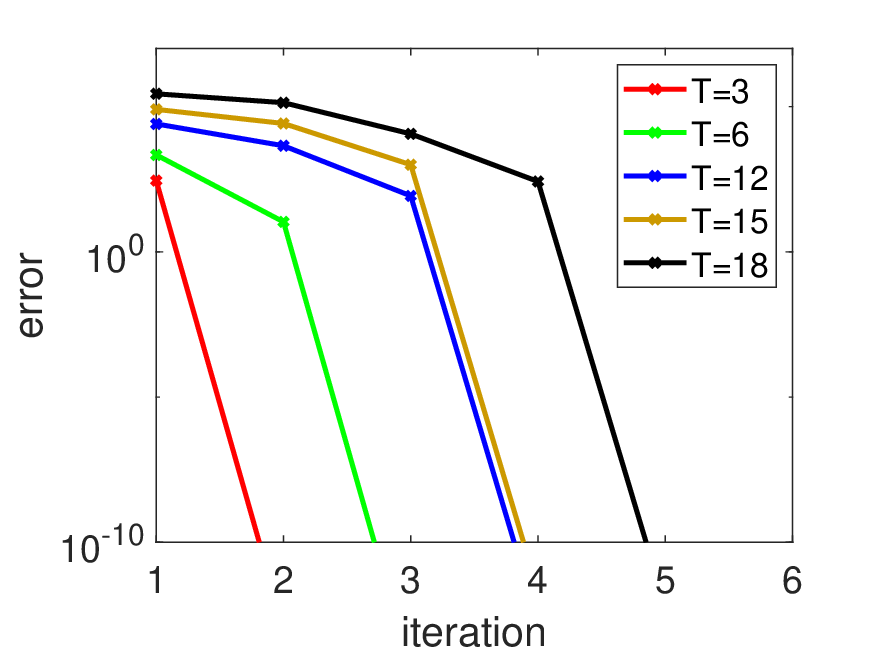}
    \caption{Convergence result of DNWR method for $\theta=1/2$ when time window (T) sizes are different: on the left - minimum subdomain width is 1.5; on the right - minimum subdomain width is 2.5 }
    \label{diftime_wave}
\end{figure}
\begin{figure}
    \centering
    \includegraphics[width=0.462\linewidth]{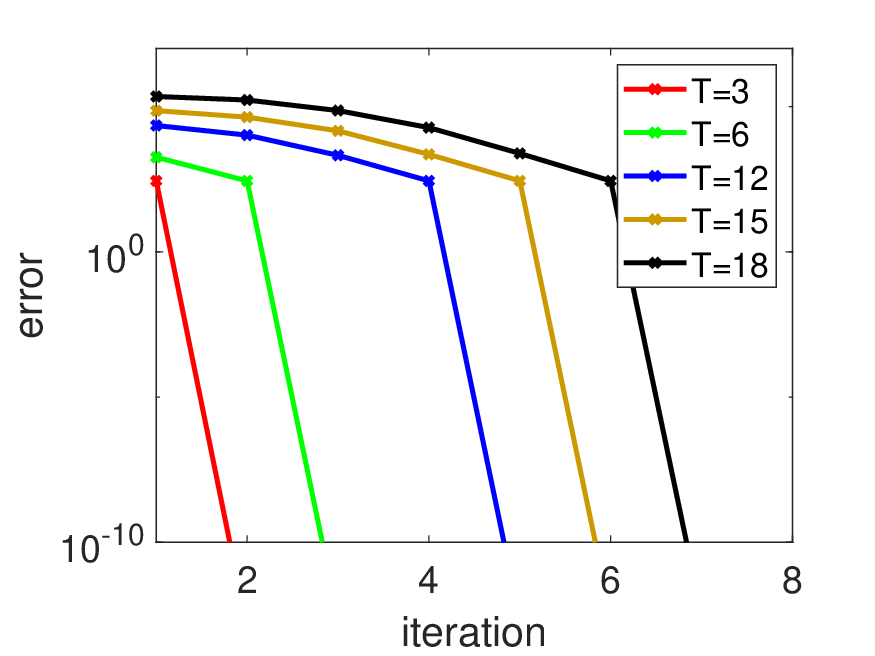}
    \includegraphics[width=0.462\linewidth]{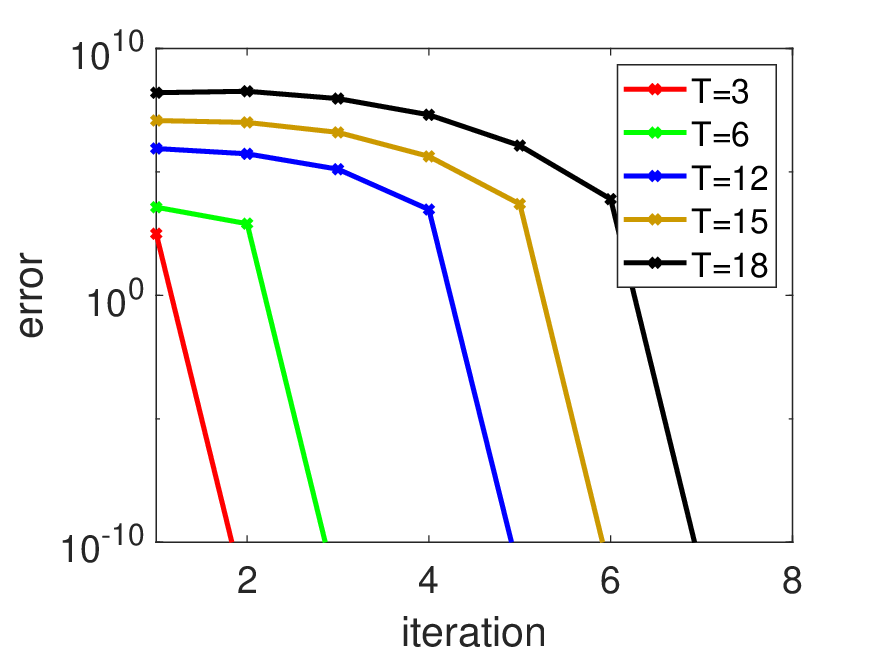}
    \caption{Convergence result of DNWR method for $\theta=1/2$ when time window (T) sizes are different and min. subdomain width is $1.5$: on the left - value of $\tau=3$; on the right - value of $\tau=0.3$ }
    \label{difdelay}
\end{figure}
We also compare the performance of DNWR and NNWR for two subdomains with the classical SWR method, using a final time of $T=6$ and an overlap of $10\times\Delta x$ in Fig.~\ref{compare_wave}. The results show that both DNWR and NNWR exhibit superior convergence behavior compared to the classical SWR method.
\begin{figure}
    \centering
    \includegraphics[width=0.48 \linewidth]{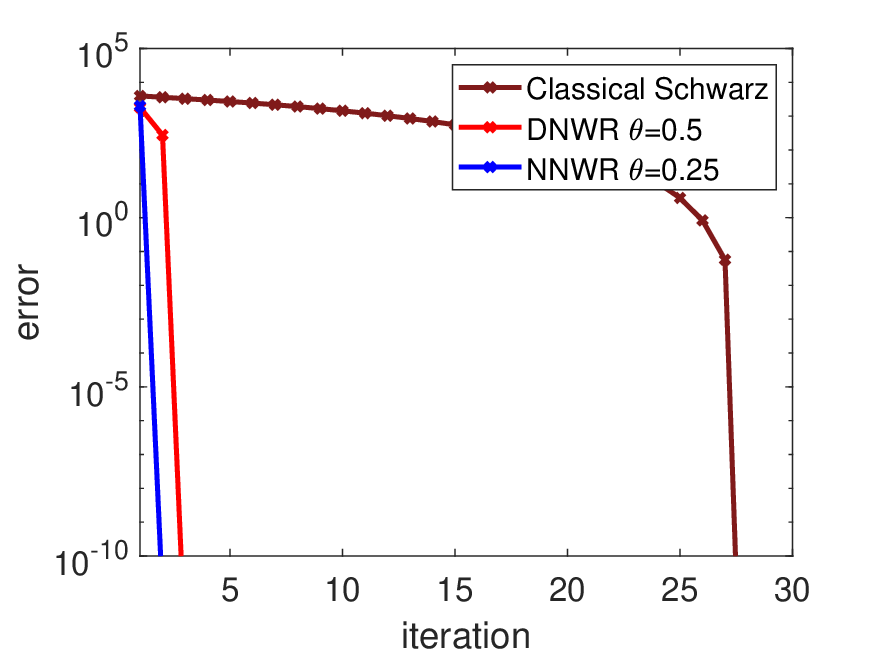}
    \caption{Comparison of DNWR and NNWR with classical Schwarz Waveform Relaxation} 
     \label{compare_wave}
    \end{figure}
\subsection{Heterogeneous Wave PDE with delay}
For the heterogeneous case, we use implicit central difference scheme and consider subdomains with lengths 
$ \left | \Omega_1\right |= 4.5$ and $\left |\Omega_2 \right|= 1.5$, and choose $\lambda = 1.6$ for the experiment. The spatial discretization is taken as \(\Delta x_1 = 0.3\) in subdomain \(\Omega_1\) and \(\Delta x_2 = 0.1\) in subdomain \(\Omega_2\),  while $\Delta t=0.1$ in both subdomains. 
We solve the delay equation with different wave propagation speeds, 
\(c_1^2 = 1\) in \(\Omega_1\) and \(c_2^2 = \frac{1}{9}\) in \(\Omega_2\), for different choices of $\theta$ (relaxation parameter).
From Fig.~\ref{hetero_wave}, we observe that the optimal value of
$\theta=0.25=1/\left (1+{\frac{c_1}{c_2}}\right )$
This result confirms the theoretical optimal value of $\theta$ for the DNWR method obtained in Theorem \ref{thm:thmdnwrhetero}.
\begin{figure}
    \centering
    \includegraphics[width=0.48 \linewidth]{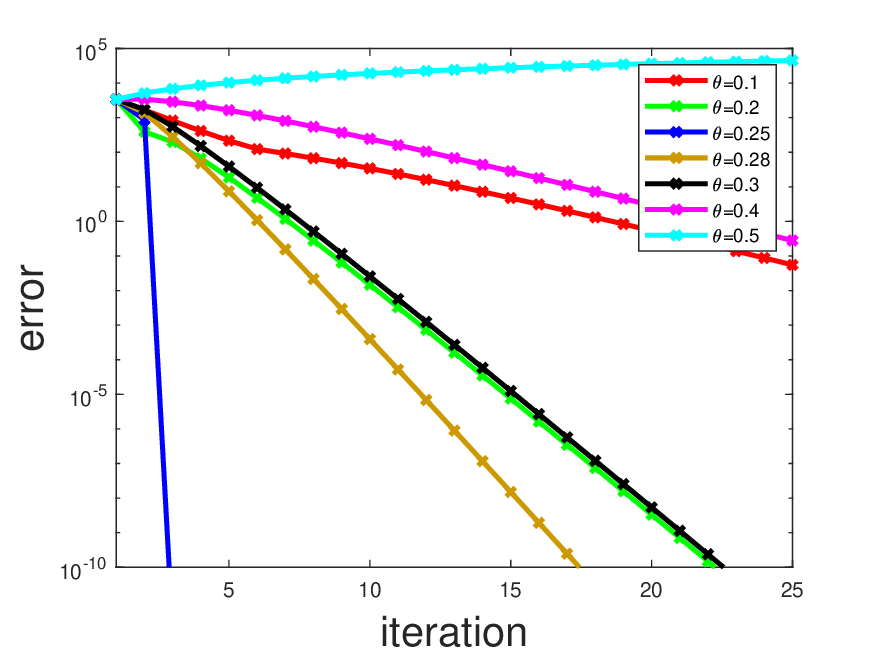}
    \caption{DNWR convergence for Wave PDE with time delay having different wave speed $c_1=1$, $c_2=1/3$} 
     \label{hetero_wave}
    \end{figure}
    \subsection{NNWR in multi-subdomain setup}
To carry out the numerical experiments in the multisubdomain setting, the spatial domain \(\Omega = (0,6)\) is first decomposed into three non-overlapping subdomains. In the first experiment, the minimum subdomain length is kept as \(h_{\min} = 1.5\), with
$\left|\Omega_1 \right |= 2.5, \left|\Omega_2\right| = 1.5, \left|\Omega_3 \right | = 2.$
In the second setting, the minimum subdomain length is \(h_{\min} = 1\), with $ \left|\Omega_1 \right| = 2, \left |\Omega_2 \right | = 1, \left |\Omega_3 \right| = 3.$
For both the cases, we use a uniform spatial mesh size \(\Delta x = 0.025= \Delta t\) (time step). See Fig.~\ref{nnwr1D_3sub} for the corresponding error plots. Next, we consider a decomposition of domain $'\Omega'$ into five subdomains, as depicts in Fig.~\ref{nnwr1D_5sub}. In the first experiment, the minimum subdomain length is \(h_{\min} = 0.5\), with $\left|\Omega_1 \right| = 1=\left|\Omega_5\right|, \quad \left|\Omega_2 \right|= 1.5, \quad \left|\Omega_3\right| = 0.5, \quad \left|\Omega_4 \right| = 2.$
In the second case, the minimum subdomain length is \(h_{\min} = 1\), where
$\left|\Omega_1\right| = \left|\Omega_3\right| = \left|\Omega_5\right| = 1,
\quad
\left|\Omega_2\right| = \left|\Omega_4\right| = 1.5.$ Fig. \ref{nnwr1D_3sub}, and Fig. \ref{nnwr1D_5sub} validates the result otained in Theorem \ref{theorem4}.
\begin{figure}
    \centering
    \includegraphics[width=0.462\linewidth]{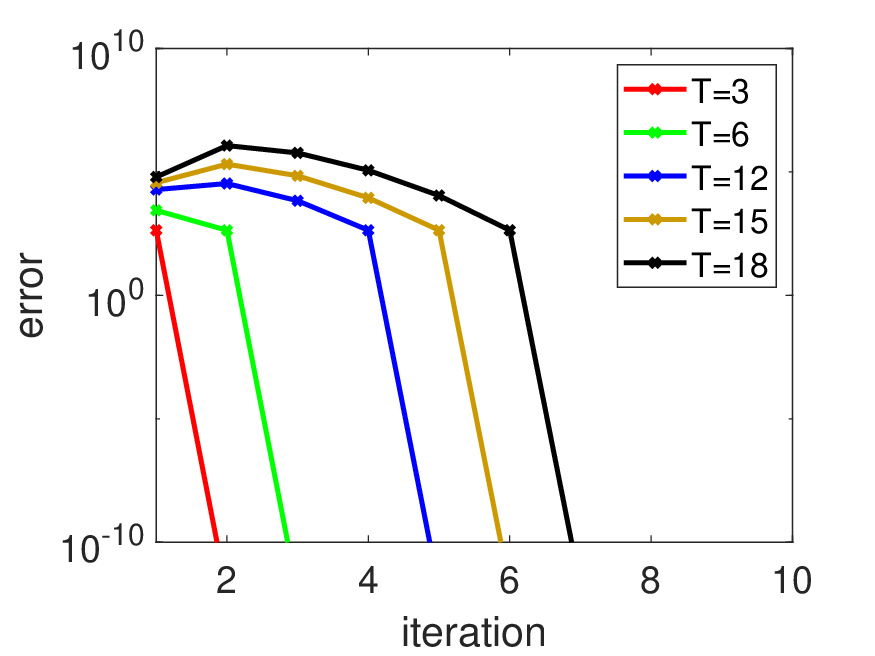}
    \includegraphics[width=0.462\linewidth]{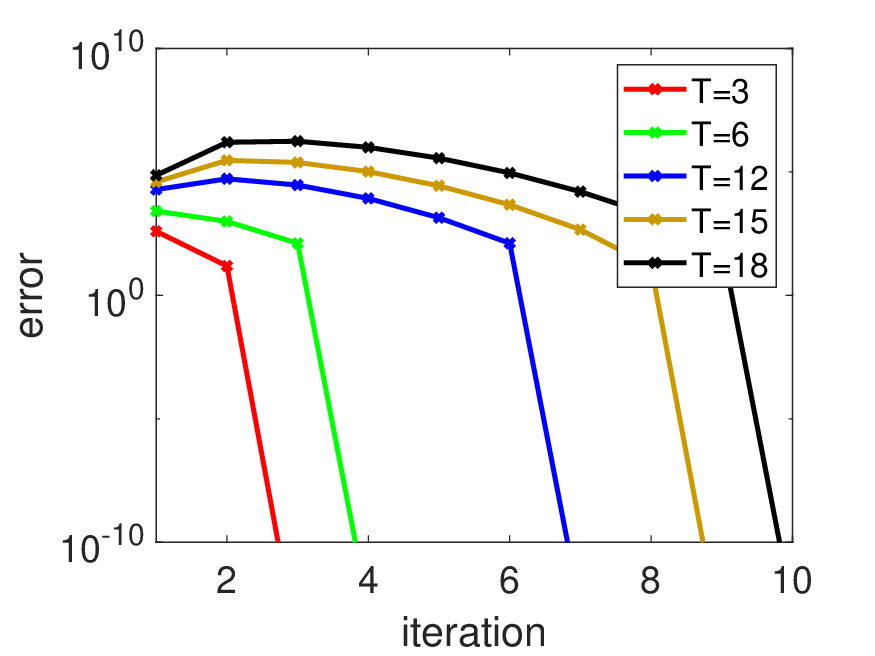}
    \caption{Convergence of NNWR methods for different time windows (T) for 3 subdomains. Left: Min subdomain length is $1.5$, Right: Min subdomain length is $1$.}
    \label{nnwr1D_3sub}
\end{figure}

\begin{figure}
    \centering
    \includegraphics[width=0.462\linewidth]{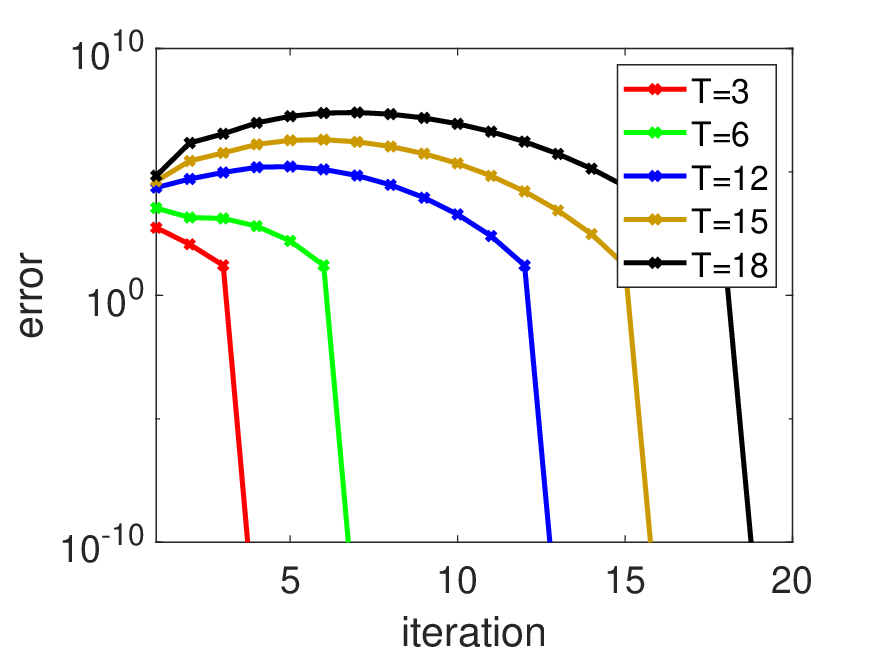}
    \includegraphics[width=0.462\linewidth]{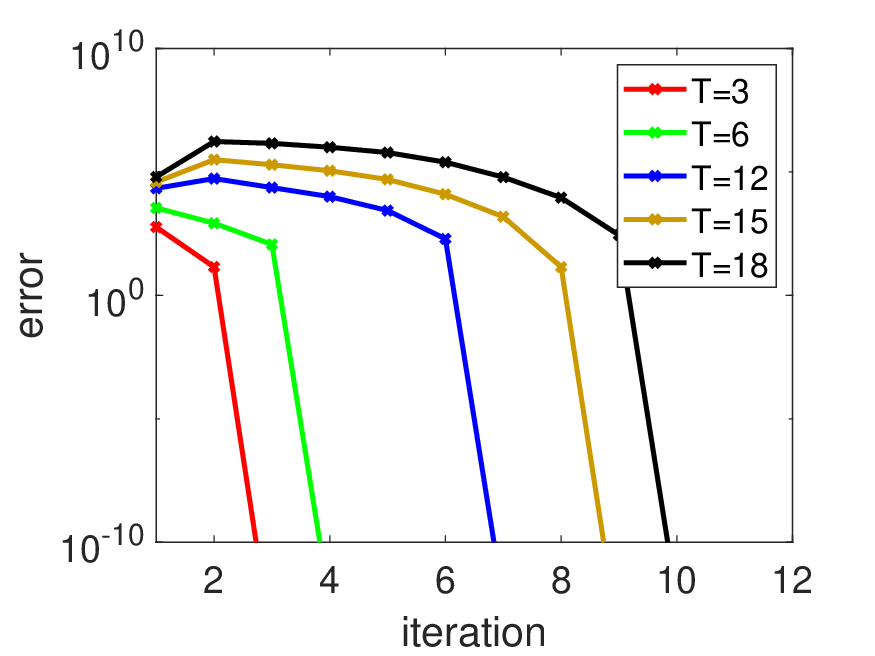}
    \caption{Convergence of NNWR methods for different time windows (T) for 5 subdomains in 1D: Left: Min subdomain length is $0.5$, Right: Min subdomain length is $1$.}
    \label{nnwr1D_5sub}
\end{figure}
\subsection{NNWR in 2D}
Consider the problem described in~\eqref{eq_1} on the two-dimensional spatial domain 
\(\Omega = (0,6)\times(0,6)\). The Leapfrog scheme is used for discretization with 
\(\Delta x = \Delta y = 0.1\) and \(\Delta t = 0.05\). The domain is partitioned into 
three non-overlapping subdomains. \textbf{Case I}, 
\(\Omega_1 = (0,0.5)\times(0,6)\), 
\(\Omega_2 = (0.5,2)\times(0,6)\), and 
\(\Omega_3 = (2,6)\times(0,6)\), with a minimum subdomain width 
\(h_{\min} = 0.5\). \textbf{Case II},
\(\Omega_1 = (0,1)\times(0,6)\), 
\(\Omega_2 = (1,3)\times(0,6)\), and 
\(\Omega_3 = (3,6)\times(0,6)\), with a minimum subdomain width 
\(h_{\min} = 1\).  The error curves are shown in Fig.~\ref{diftime_wave2D} confirms our theoretical findings from Theorem \ref{thm5}.

\begin{figure}
    \centering
    \includegraphics[width=0.462\linewidth]{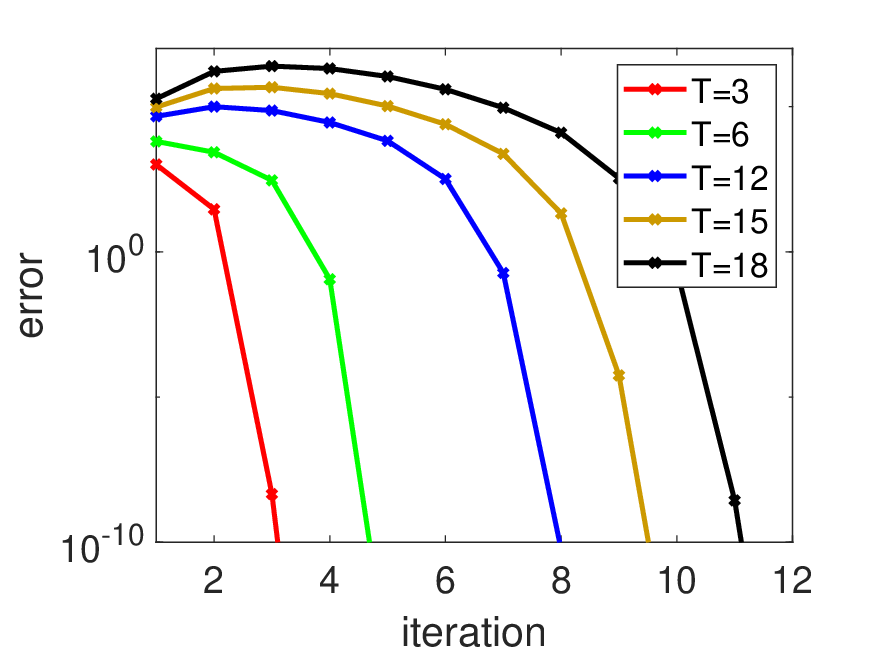}
    \includegraphics[width=0.462\linewidth]{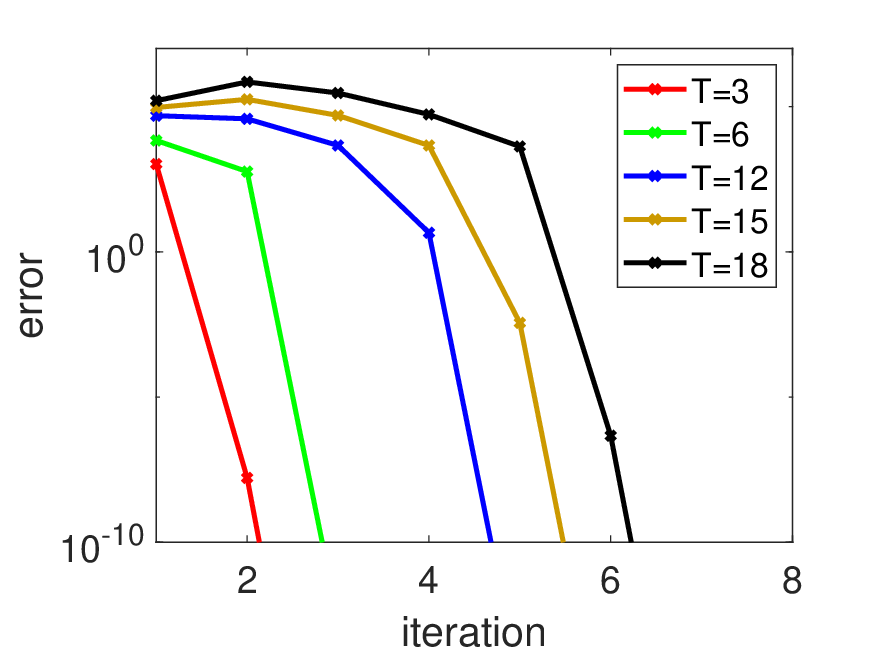}
    \caption{Convergence of NNWR methods for different time windows (T) for 3 subdomains in 2D. Left: Min subdomain length is $0.5$, Right: Min subdomain length is $1$.}
    \label{diftime_wave2D}
\end{figure}

\section{Conclusion}

This paper provides rigorous convergence analysis of the DNWR method for 
asymmetrical domain decomposition using both Fourier and Laplace transform 
techniques, applied to the wave equation with time delay. The Fourier analysis of the DNWR algorithm for asymmetric subdomains establishes that for a relaxation parameter $\theta=1/2$, the method exhibits linear convergence. Furthermore, using Laplace transform analysis, we derive the finite-step convergence property. Significant results are also obtained for the heterogeneous media; satisfying condition $|\Omega_1|/{c_1} = |\Omega_2|/{c_2}$, the DNWR method achieves finite-step convergence for the optimal parameter $\theta = 1/(1+{c_1/c_2})$.\\

We also analyze the NNWR method in a multi-subdomain setting in 1-D, which is extended to two-dimensional 
spatial domains decomposed into strips. We establish that with a relaxation parameter fixed at $\theta=1/4$, the method achieves the finite step convergence. Numerical simulations are presented to validate the theoretical findings. The 
outcomes demonstrate that both DNWR and NNWR methods are effective for solving 
wave equations with time delay, and can be treated as a two-iteration methods for sufficiently small time window lengths. These findings suggest that the proposed framework can be effectively generalized to more complex systems, offering a promising solution for parallel computing for delay differential equations.\\
\bibliographystyle{elsarticle-harv}
\bibliography{references}
\end{document}